\documentclass[12pt]{amsart}
\usepackage{amsfonts,amssymb,amscd,amstext,mathrsfs}

\input xy
\xyoption{all}

\newtheorem{theorem}{Theorem}[section]
\newtheorem{lemma}[theorem]{Lemma}
\newtheorem{corollary}[theorem]{Corollary}
\newtheorem{proposition}[theorem]{Proposition}

\theoremstyle{definition}

\newtheorem{example}[theorem]{Example}

\newtheorem{exercise}[theorem]{Exercise}

\newcommand{\C}{\mathbb{C}}
\newcommand{\D}{\mathbb{D}}
\newcommand{\N}{\mathbb{N}}
\renewcommand{\P}{\mathbb{P}}
\newcommand{\R}{\mathbb{R}}

\newcommand{\cC}{\mathscr{C}}
\newcommand{\cO}{\mathscr{O}}
\newcommand{\id}{\mathrm{id}}
\newcommand{\Hom}{\operatorname{Hom}}
\newcommand{\Ker}{\operatorname{Ker}}
\renewcommand{\Re}{\operatorname{Re}}

\numberwithin{equation}{section}

\begin{document}

\title{Eight lectures on Oka manifolds}

\author{Finnur L\'arusson}

\address{Finnur L\'arusson, School of Mathematical Sciences, University of Adelaide, Adelaide SA 5005, Australia}
\email{finnur.larusson@adelaide.edu.au}

\dedicatory{Notes for lectures given at the Institute of Mathematics \\ of the Chinese Academy of Sciences in Beijing in May 2014}

\thanks{The author would like to thank Professor Xiangyu Zhou and Professor John Erik Forn\ae ss for the invitation to give these lectures.  The author would also like to thank Professor Zhou and the Institute of Mathematics of the Chinese Academy of Sciences for their kind hospitality.}

\thanks{The author was supported by Australian Research Council grants DP120104110 and DP150103442.}

\date{28 May 2014 (last minor changes 6 October 2015)}

\maketitle
\tableofcontents

%  INTRODUCTION

\section{Introduction} 
\label{s:intro}

\noindent
Over the past decade, the class of Oka manifolds has emerged from Gromov's seminal work on the Oka principle.  Roughly speaking, Oka manifolds are complex manifolds that are the target of \lq\lq many\rq\rq\ holomorphic maps from affine spaces.  They are \lq\lq dual\rq\rq\ to Stein manifolds and \lq\lq opposite\rq\rq\ to Kobayashi-hyperbolic manifolds.  The prototypical examples are complex homogeneous spaces, but there are many other examples: there are many ways to construct new Oka manifolds from old.  The class of Oka manifolds has good formal properties, partly explained by a close connection with abstract homotopy theory.  These notes are meant to give an accessible introduction, not to all of Oka theory, but more specifically to Oka manifolds, how they arise and what we know about them.

In this section, we will see two different ways to approach and motivate the concept of an Oka manifold.  First let us mention a few references.  Modern Oka theory began with Gromov's 1989 paper \cite{Gromov1989}.  There are three recent surveys, \cite{Forstneric2013},  \cite{ForstnericLarusson2011}, and \cite{Larusson2010a}, and a very comprehensive monograph \cite{Forstneric2011}.  The notes \cite{Larusson2009a} give an elementary introduction to the notions of ellipticity and hyperbolicity in complex analysis.

\subsection{How flexible are holomorphic maps?}  Let $X$ and $Y$ be complex manifolds (in these notes always taken to be connected and second countable).  Can every continuous map $X\to Y$ be deformed to a holomorphic map?  That is, does every homotopy class of continuous maps $X\to Y$ contain a holomorphic map?  Let us look at some examples.

\begin{itemize}
\item  $\C^*\to\D^*$: every holomorphic map is constant (Liouville), so the only winding number realised by holomorphic maps is $0$.  Here, $\D$ is the disc $\{z\in\C:\lvert z\rvert<1\}$.
\item  $\D^*\to\D^*$: holomorphic maps only realise nonnegative winding numbers.
\item  $\D\setminus\{\tfrac 1 2, \tfrac 1 3, \tfrac 1 4, \ldots, 0\}\to\D^*$: there are uncountably many homotopy classes of continuous maps, but only countably many classes of holomorphic maps.
\end{itemize}
In all three examples, if the target is $\C^*$ instead of $\D^*$, then every continuous map is homotopic to a holomorphic map.  More generally:

\begin{theorem}  \label{t:C-star}
Let $X$ be a Stein manifold and $f:X\to\C^*$ be a continuous map.  Then $f$ is homotopic to a holomorphic map.
\end{theorem}

More explicitly, the conclusion of the theorem is that there is a continuous map $F:X\times I\to\C^*$, where $I=[0,1]$, such that $F(\cdot,0)=f$ and $F(\cdot,1)$ is holomorphic.

We will discuss Stein manifolds in Section \ref{s:stein}.  If you are not familiar with Stein manifolds, take $X$ to be a domain of holomorphy in $\C^n$, or an open Riemann surface, or simply a domain in $\C$.

The theorem is easily proved for domains in $\C$ whose fundamental group is finitely generated.  Let us only consider the example $X=\C\setminus\{a_1,\dots,a_n\}$, where $a_1,\dots,a_n$ are distinct points.  The homotopy class of a continuous map $f:X\to\C^*$ is determined by how many times the image by $f$ of a little circle around each puncture wraps around the origin.  If we denote this winding number for $a_j$ by $k_j\in\mathbb Z$, then $f$ is homotopic to the holomorphic map $z\mapsto \prod\limits_{j=1}^n(z-a_j)^{k_j}$.  

\begin{proof}[Proof of Theorem \ref{t:C-star}]
Each point of $X$ has a neighbourhood on which $f$ has a continuous logarithm, so there is an open cover $(U_\alpha)$ of $X$ such that $f=e^{2\pi i\lambda_\alpha}$ on $U_\alpha$ for some continuous $\lambda_\alpha:U_\alpha\to\C$ (we throw in the factor $2\pi i$ for convenience and suppress the index set through which $\alpha$ runs).  Then, for all $\alpha, \beta$, the function $n_{\alpha\beta}=\lambda_\alpha-\lambda_\beta:U_{\alpha\beta}=U_\alpha\cap U_\beta\to\mathbb Z$ is locally constant because it is continuous and takes values in $\mathbb Z$.  Clearly, $n_{\alpha\beta}+n_{\beta\gamma}=n_{\alpha\gamma}$ on $U_\alpha\cap U_\beta\cap U_\gamma$ (this is called a \textit{cocycle condition}).

Suppose that we could find \textit{holomorphic} functions $\mu_\alpha:U_\alpha\to\C$ such that $n_{\alpha\beta}=\mu_\alpha-\mu_\beta$ on $U_{\alpha\beta}$ for all $\alpha,\beta$.  (We already have such a \textit{splitting} of the \textit{cocycle} $(n_{\alpha\beta})$ by the \textit{continuous} functions $\lambda_\alpha$.)  Then we would get a well-defined holomorphic function $g:X\to\C^*$ by setting $g=e^{2\pi i\mu_\alpha}$ on $U_\alpha$, and the formula
\[ F(x,t)=\exp\big(2\pi i\big((1-t)\lambda_\alpha(x)+t\mu_\alpha(x)\big)\big) \qquad\text{for }(x,t)\in U_\alpha\times I \]
would define a continuous map $F:X\times I\to\C^*$ with $F(\cdot,0)=f$ and $F(\cdot,1)=g$, as desired.

So how do we split the cocycle $(n_{\alpha\beta})$ holomorphically?  First we split it smoothly.  We choose a partition of unity $(\phi_\alpha)$ subordinate to the open cover $(U_\alpha)$.  This means that for each $\alpha$ we have a smooth function $\phi_\alpha:X\to I$ such that:
\begin{itemize} 
\item  the support $\overline{\{x\in X:\phi_\alpha(x)\neq 0\}}$ of $\phi_\alpha$ is contained in $U_\alpha$,
\item  every point of $X$ has a neighbourhood on which $\phi_\alpha$ is identically zero for all but finitely many $\alpha$,
\item  $\sum\limits_\alpha\phi_\alpha=1$ on $X$.
\end{itemize}
For each $\alpha, \gamma$, we note that $n_{\alpha\gamma}\phi_\gamma$ is a well-defined smooth function on $U_\alpha$, so the sum $\nu_\alpha=\sum\limits_\gamma n_{\alpha\gamma}\phi_\gamma$ is a well-defined smooth function on $U_\alpha$.  Also, by the cocycle condition,
\[ \nu_\alpha-\nu_\beta = \sum_\gamma(n_{\alpha\gamma}-n_{\beta\gamma})\phi_\gamma = \sum_\gamma n_{\alpha\beta}\phi_\gamma = n_{\alpha\beta} \sum_\gamma \phi_\gamma = n_{\alpha\beta} \qquad\text{on }U_{\alpha\beta}.\]
From this smooth splitting we can obtain a holomorphic splitting.  Namely, since $n_{\alpha\beta}$ is locally constant, we have $\bar\partial\nu_\alpha-\bar\partial\nu_\beta=\bar\partial n_{\alpha\beta}=0$ on $U_{\alpha\beta}$, so we get a well-defined smooth $(0,1)$-form $\theta$ on $X$, defined as $\bar\partial\nu_\alpha$ on $U_\alpha$ for each $\alpha$.  Since $X$ is Stein, there is a smooth function $u:X\to\C$ with $\bar\partial u = \theta$.  Finally, set $\mu_\alpha=\nu_\alpha-u$.  Then $\mu_\alpha$ is holomorphic on $U_\alpha$ and $\mu_\alpha-\mu_\beta=(\nu_\alpha-u)-(\nu_\beta-u)=n_{\alpha\beta}$ on $U_{\alpha\beta}$.
\end{proof}

We used the exponential map $\C\to\C^*$ to linearise the problem.  The exponential map allowed us to reduce the problem to solving the linear equation $\bar\partial u=\theta$, which we can do on a Stein manifold.

If $Y$ is a complex manifold and every continuous map from a Stein manifold into $Y$ is homotopic to a holomorphic map, then we say that $Y$ satisfies the \textit{basic Oka property}, abbreviated BOP.  Theorem \ref{t:C-star} thus says that $\C^*$ satisfies BOP.

\subsection{Approximation and interpolation problems}  Here are two well-known theorems from 19th century complex analysis, concerning a domain $X$ in $\mathbb C$.  

\noindent
\textit{Weierstrass theorem.}  If $S$ is a discrete subset of $X$, then every function $S\to\C$ extends to a holomorphic function $X\to\C$.

\noindent
\textit{Runge approximation theorem.}  If $K$ is a compact subset of $X$ with no holes in $X$,\footnote{By a \textit{hole} of $K$ in $X$ we mean a connected component of $X\setminus K$ that is relatively compact in $X$.} then every holomorphic function $K\to\C$ can be uniformly approximated on $K$ by holomorphic functions $X\to\C$.  (By a holomorphic function $K\to\C$ we mean a holomorphic function on some neighbourhood of $K$.)  

In the formative years of modern complex analysis, in the mid-20th century, these theorems were extended to Stein manifolds $X$ of arbitrary dimension.  The \textit{Cartan extension theorem} generalises $S$ to a closed analytic subvariety of $X$ and says that every holomorphic function $S\to\C$ extends to a holomorphic function $X\to\C$.  The \textit{Oka-Weil approximation theorem} replaces the topological condition that $K$ have no holes in $X$ with the subtle, non-topological condition that $K$ be \textit{holomorphically convex} in $X$ or $\cO(X)$-convex.  This means that for every $x\in X\setminus K$, there is a holomorphic function $f$ on $X$ with $\lvert f(x)\rvert > \max_K\lvert f\rvert$.

We usually consider these theorems as results about Stein manifolds, and of course they are, but we can also view them as expressing properties of the target $\C$.  \textit{We can then formulate them for a general target.}  We need to keep in mind that there may be topological obstructions to solving approximation and interpolation problems.  There is no point in asking for an analytic solution if there is no continuous solution.

Let us say that a complex manifold $Y$ satisfies the \textit{interpolation property} (IP) if for every Stein manifold $X$ with a subvariety $S$, a holomorphic map $S\to Y$ has a holomorphic extension $X\to Y$ if it has a continuous extension.

Say that $Y$ satisfies the \textit{approximation property}\footnote{The terms IP and AP are not standard, although they are very natural.} (AP) if for every Stein manifold $X$ with a holomorphically convex compact subset $K$, a continuous map $X\to Y$ that is holomorphic on $K$ can be uniformly approximated on $K$ by holomorphic maps $X\to Y$.

We can formulate ostensibly much weaker versions of these properties, for very special $X$, $S$, and $K$.

We say that $Y$ satisfies the \textit{convex interpolation property}\footnote{For the reason for using the word \lq\lq convex\rq\rq\ here, see \cite[\S 4]{Larusson2010b}, where this property was introduced.}
(CIP) if, whenever $S$ is a contractible subvariety of $\C^m$ for some $m$, every holomorphic map $S\to Y$ extends to a holomorphic map $\C^m\to Y$.

We say that $Y$ satisfies the \textit{convex approximation property}\footnote{This property was introduced by Forstneri\v c in \cite{Forstneric2006}.} (CAP) if, whenever $K$ is a convex compact subset of $\C^m$ for some $m$, every holomorphic map $K\to Y$ can be uniformly approximated on $K$ by holomorphic maps $\C^m\to Y$.  

So which manifolds satisfy these generalised Cartan and Oka-Weil theorems?  Are there any examples besides $\C$, $\C^2$, $\C^3$, \ldots ?  The following theorem is the only reasonably easy result in this direction.

\begin{theorem}  \label{t:lie-groups}
Every complex Lie group satisfies CAP.
\end{theorem}

\begin{proof}
(Following \cite[Proposition 5.5.1]{Forstneric2011}.)  Let $G$ be a complex Lie group.  The exponential map $\exp:\mathfrak g=T_eG\to G$ is a local biholomorphism at the origin in $\mathfrak g$.  Let $f:U\to G$ be a holomorphic map, where $U$ is a convex neighbourhood of a convex compact subset $K$ of $\C^m$.

If $f(K)$ lies close enough to $e$, then there is a holomorphic map $h:V\to\mathfrak g$, defined on a smaller neighbourhood $V$ of $K$, such that $f=\exp h$.  Approximating $h$ uniformly on $K$ by a holomorphic map $\tilde h:\C^m\to\mathfrak g$ using Oka-Weil gives the approximation $\exp\tilde h:\C^m\to G$ of $f$.

In general, we may assume that $0\in K$ and write $f_t(z)=f(tz)$ for $t\in I$ and $z\in U$.  Then $f_1=f$ and $f_0$ is constant.  For $n\in\N$, use the group structure of $G$ to write
\[ f=f_1\cdot (f_{\frac{n-1}n})^{-1}\cdot f_{\frac{n-1}n} \cdot (f_{\frac{n-2}n})^{-1} \cdots f_{\frac 1 n}\cdot (f_0)^{-1}\cdot f_0. \]
If $n$ is sufficiently large, then each quotient $f_{\frac k n}\cdot (f_{\frac {k-1} n})^{-1}$, $k=1,\ldots,n$, is close enough to $e$ on $K$ that it admits a holomorphic logarithm $h_k:K\to\mathfrak g$.  Approximating each $h_k$ uniformly on $K$ by a holomorphic map $\tilde h_k:\C^m\to\mathfrak g$ gives the approximation
\[ \exp\tilde h_n\cdots\exp\tilde h_1\cdot f_0:\C^m\to G \]
of $f$.
\end{proof}

Again we used the exponential map to linearise the problem.  A similar but somewhat more elaborate proof works for homogeneous manifolds.  Gromov's main innovation in \cite{Gromov1989} was a more general linearisation method that allows us to establish CAP for a considerably larger class of manifolds, the so-called elliptic manifolds (see Section \ref{s:linearisation}).

What about relationships between the properties we have defined?  There is a short proof that interpolation implies approximation \cite{Larusson2005}.  (This is not to say that the Cartan extension theorem implies the Oka-Weil approximation theorem: the proof that IP implies AP uses the Oka-Weil theorem.)

\begin{theorem}  \label{t:ip-implies-ap}
If a complex manifold satisfies IP, then it also satisfies AP.
\end{theorem}

Similarly, CIP implies CAP.

\begin{proof}
Suppose $Y$ satisfies IP.  Let $X$ be a Stein manifold, $K\subset X$ be compact and holomorphically convex, and $f:X\to Y$ be continuous, and holomorphic on a Stein neighbourhood $U$ of $K$.  (A holomorphically convex compact subset of a Stein manifold has a Stein neighbourhood basis.)  Let $\phi:U\to\C^n$ be a holomorphic embedding (see Theorem \ref{t:stein}; by an embedding we always mean a proper embedding).  The inclusion $i:U\hookrightarrow X$ factors through the Stein manifold $M=X\times\C^n$ as $U\stackrel{j}{\to} M \stackrel{\pi}{\to} X$, where $j=(i,\phi)$ is an embedding and $\pi$ is the projection.

Now $f\circ\pi:M\to Y$ is continuous, and holomorphic on $j(U)$.  Since $Y$ satisfies IP, there is a holomorphic extension $h:M\to Y$ of $f\circ\pi\vert j(U)$.  Since $K$ is holomorphically convex in $X$, by Oka-Weil we can approximate $\phi:U\to\C^n$ uniformly on $K$ by a holomorphic map $\psi:X\to\C^n$.  Then $h\circ (\id_X, \psi):X\to Y$ is holomorphic and approximates $f$ uniformly on $K$.
\end{proof}

In fact, the converse is true---AP implies IP---but this is not easy to see.  It is a major theorem of Forstneri\v c that CAP implies IP, so CAP, CIP, AP, and IP are equivalent, and that CAP implies BOP (see Section \ref{s:forstneric}).  In fact, Forstneri\v c proved that CAP implies a property that is stronger than both IP and BOP and is defined as follows.

We say that a complex manifold $Y$ satisfies the \textit{basic Oka property with approximation and interpolation} (BOPAI) if for every Stein manifold $X$ with a subvariety $S$ and a holomorphically convex compact subset $K$, a continuous map $f:X\to Y$ that is holomorphic on $S$ and on $K$ can be deformed to a holomorphic map, keeping it fixed on $S$ and almost fixed on $K$.  More precisely,  for every $\epsilon>0$, there is a continuous map $F:X\times I\to Y$ such that:
\begin{itemize}
\item $F(\cdot,0)=f$,
\item $F(\cdot,1)$ is holomorphic,
\item $F(x,t)=f(x)$ for all $x\in S$ and $t\in I$,
\item $d(F(x,t), f(x))<\epsilon$ for all $x\in K$ and $t\in I$.
\end{itemize}
Here, $d$ is any metric defining the topology of $Y$.

We now define an \textit{Oka manifold} to be a complex manifold satisfying the equivalent properties CAP, CIP, AP, IP, and BOPAI.\footnote{Oka manifolds were first defined in \cite{Forstneric2009} and \cite[\S 16]{Larusson2004}.  The concept of an Oka manifold is unusual in that a big theorem is needed in order to define it in a satisfactory manner.}

By Theorem \ref{t:lie-groups}, every complex Lie group (and every homogeneous manifold) is Oka.  More generally, every elliptic manifold is Oka (Corollary \ref{t:ell-implies-Oka}).

Note that BOP is BOPAI with $S$ and $K$ empty.  

\begin{exercise}  \label{x:disc}
Show that the disc $\D$ satisfies BOP but is not Oka.
\end{exercise}

%  STEIN MANIFOLDS

\section{Stein manifolds}  
\label{s:stein}

\noindent
It follows easily from the Weierstrass theorem that every domain in $\C$ is a domain of holomorphy: it carries a holomorphic function that does not extend to any larger domain (even as a multivalued function).  This is far from true in higher dimensions.  For example, every holomorphic function $f$ on $\C^2\setminus\{0\}$ extends to a holomorphic function on $\C^2$ (Hurwitz 1897).  Namely, the function $\C\times\D\to\C$ defined by
\[ (z,w) \mapsto \frac 1{2\pi i}\int_{\lvert\zeta\rvert=1}\frac{f(z,\zeta)}{\zeta-w}\,d\zeta \]
is holomorphic and agrees with $f$ on $\C^*\times\D$ by Cauchy's formula, so it agrees with $f$ on $\C\times\D\setminus\{0\}$.  By a more elaborate, but still elementary, application of Cauchy's formula, one can show that if $K$ is a compact subset of a domain $\Omega$ in $\C^n$, $n\geq 2$, such that $\Omega\setminus K$ is connected, then every holomorphic function on $\Omega\setminus K$ extends to a holomorphic function on $\Omega$ \cite{Sobieszek2003}.

In higher dimensions, the notion of a domain of holomorphy is quite subtle and is indeed the central concept of classical several complex variables.  The notion of a Stein manifold, introduced by Stein in 1951, generalises domains of holomorphy to the setting of complex manifolds.

A complex manifold $X$ is said to be \textit{Stein} if it satisfies the following three conditions.
\begin{enumerate}
\item[(a)]  Holomorphic functions on $X$ \textit{separate points}, that is, if $x,y\in X$, $x\neq y$, then there is $f\in\cO(X)$ such that $f(x)\neq f(y)$.  Here, $\cO(X)$ denotes the algebra of holomorphic functions on $X$.
\item[(b)]  Holomorphic functions on $X$ \textit{separate directions}, that is, for every $x\in X$, there are $f_1,\ldots,f_n\in\cO(X)$ that give coordinates at $x$ (so $n=\dim X$). 
\item[(c)]  $X$ is \textit{holomorphically convex}, that is, if $K\subset X$ is compact, then its $\cO(X)$-hull
\[ \hat K = \{ x\in X:\lvert f(x)\rvert\leq \max_K\lvert f\rvert \textrm{ for all }f\in\cO(X)\} \]
is also compact.  Equivalently, if $E\subset X$ is not relatively compact, then there is $f\in\cO(X)$ such that $f\vert E$ is unbounded.
\end{enumerate}
There is some redundancy among these properties.  It may be shown that (a) follows from (b) and (c), and that (b) follows from (a) and (c).

Just as every domain in $\C$ is a domain of holomorphy, every open Riemann surface is Stein.

As befits a notion of fundamental importance, there are several alternative characterisations of the Stein property.

\begin{theorem}  \label{t:stein}
For a complex manifold $X$, the following are equivalent.
\begin{enumerate}
\item[(i)]  $X$ is Stein.
\item[(ii)]  $X$ is biholomorphic to a closed complex submanifold of $\C^m$ for some $m$.
\item[(iii)]  $X$ is {\em strictly pseudoconvex}, that is, there is a smooth strictly plurisubharmonic function $\psi:X\to\R$  that is an exhaustion in the sense that for every $c\in\R$, the sublevel set $\{x\in X:\psi(x)<c\}$ is relatively compact in $X$.  The closure of each sublevel set is then holomorphically convex.
\item[(iv)]  $H^p(X,\mathscr F)=0$ for every coherent analytic sheaf $\mathscr F$ on $X$ and every $p\geq 1$.
\end{enumerate}
\end{theorem}

The implication (i) $\Rightarrow$ (ii) (with $m=2\dim X+1$) is the embedding theorem of Remmert, Bishop, and Narasimhan.  The implication (i) $\Rightarrow$ (iv) is Cartan's Theorem B.

\begin{proof}[Brief comments on proofs.]
(i) $\Rightarrow$ (ii):  H\"ormander's account in \cite[\S 5.3]{Hormander1979} has not been improved upon.

(ii) $\Rightarrow$ (iii):  This is easy: let $\psi(x)=\lVert x\rVert^2$, where $\lVert\cdot\rVert$ is the Euclidean norm on $\C^m$.

(iii) $\Rightarrow$ (iv):  This is a major theorem.  A proof using the $L^2$ method of H\"ormander and of Andreotti and Vesentini is given in \cite[Corollary IX.4.11]{Demailly2012} and \cite[Theorem 7.4.3]{Hormander1979}.  For a very readable introduction to the $L^2$ method in complex analytic and differential geometry, see \cite{Berndtsson2010}.

(iv) $\Rightarrow$ (i):  To verify (a), for example, let $\mathscr I$ be the sheaf of holomorphic functions on open subsets of $X$ that vanish at $x$ and $y$.  Then $\mathscr I$ is a coherent sheaf of ideals in the structure sheaf $\cO$ of $X$.  The quotient $\cO/\mathscr I$ is a \lq\lq skyscraper sheaf\rq\rq\ with stalk $\C$ at $x$ and $y$ and zero elsewhere.  The short exact sequence $0\to\mathscr I\to\cO\to\cO/\mathscr I\to 0$ gives the long exact sequence $\cdots\to\cO(X)\to \C^2 \to H^1(X,\mathscr I)\to\cdots$.  By (iv), $\cO(X)\to \C^2$ is surjective, so we can prescribe the values of holomorphic functions on $X$ at $x$ and~$y$.  The proof of (b) and (c) is similar.
\end{proof}

It is natural to ask for the smallest $m$ that works in (ii).  Forster conjectured that if $n=\dim X\geq 2$, then the smallest possible $m$ is $[3n/2]+1$.  He showed by examples that for each $n\geq 2$, no smaller value of $m$ works in general.  Forster's conjecture was proved in the early 1990s by Eliashberg and Gromov and by Sch\"urmann.  The proof relies on Gromov's Oka principle.  It is still a wide open question whether every open Riemann surface embeds in $\C^2$.

Cartan's Theorem B is a very powerful result.  For example, it implies the Cartan extension theorem (that is, IP for $\C$) as soon as we know that the sheaf of holomorphic functions that vanish on the subvariety $S$ is coherent (this is a fairly deep theorem of local analytic geometry).  Here is another consequence that will be useful later.

\begin{proposition}  \label{p:vector-bundles}
(a)  Every short exact sequence of holomorphic vector bundles on a Stein manifold splits.

(b)  Every holomorphic vector bundle on a Stein manifold is a direct summand in a trivial vector bundle.
\end{proposition}

\begin{proof}
(a)  That a short exact sequence $0\to K\to F\to E\to 0$ of vector bundles on a Stein manifold $X$ splits follows from the long exact sequence
\[ \cdots\to\Hom (E,F)\to\Hom (E,E)\to H^1(X,\mathscr H\textit{\!om}(E,K))=0 \to\cdots. \]

(b)  Let $E$ be a vector bundle on $X$.  Then $E$ is a quotient of a trivial bundle $F$.  In other words, there are finitely many sections of $E$ that generate $E$ at each point.  The proof is similar to the proof that (iv) $\Rightarrow$ (i) in Theorem \ref{t:stein} above.  Now apply (a).
\end{proof}

For a proof of the Oka-Weil approximation theorem (that is, AP for $\C$) using the $L^2$ method, see \cite[Corollary 5.2.9]{Hormander1979}.  By Proposition \ref{p:vector-bundles}(b), Oka-Weil approximation holds for sections of vector bundles over Stein manifolds (and so does Cartan extension).

In the setting of Proposition \ref{p:vector-bundles}(a), a splitting morphism $E\to F$ over an open subset of $X$ is a section of an affine bundle, which has a section over $X$ and is thus a vector bundle.  Hence Oka-Weil approximation holds for splitting morphisms.

%  GROMOV'S LINEARISATION METHOD

\section{Gromov's linearisation method}
\label{s:linearisation}

\noindent
In his seminal paper of 1989 \cite{Gromov1989}, Gromov introduced a useful geometric structure that generalises the exponential map of a complex Lie group and allows us to establish the convex approximation property.

\subsection{Elliptic manifolds}  A \textit{spray} on a complex manifold $Y$ is a holomorphic map $s:E\to Y$ defined on the total space of a holomorphic vector bundle $E$ over $Y$ such that $s(0_y)=y$ for all $y\in Y$.  The spray is said to be \textit{dominating at} $y\in Y$ if $s\vert E_y\to Y$ is a submersion at $0_y$.  The spray is said to be \textit{dominating} if it is dominating at every point of $Y$.  Finally, $Y$ is said to be \textit{elliptic} if it admits a dominating spray.

\begin{example}  \label{e:elliptic}
(a)  If $Y$ is homogeneous, that is, there is a transitive holomorphic action of a complex Lie group $G$ on $Y$, then $Y$ is elliptic.  The map $Y\times\mathfrak g\to Y$, $(y,v)\mapsto \exp(v)\cdot y$, where $\mathfrak g$ is the Lie algebra of $G$, is a dominating spray defined on a trivial vector bundle over $Y$.  In particular, a complex Lie group is elliptic.

(b)  More generally, if $Y$ carries finitely many $\C$-complete holomorphic vector fields $v_1,\ldots,v_k$ that span $T_yY$ at each $y\in Y$, then the map $s:Y\times\C^k\to Y$,
\[ s(y,t_1,\ldots,t_k)=\phi_k^{t_k}\circ\cdots\circ\phi_1^{t_1}(y), \]
where $\phi_j^t$ is the flow of $v_j$, is a dominating spray on $Y$.  Note that
\[ \frac\partial{\partial t_j}s(y,0) = v_j(y), \qquad y\in Y. \]

(c)  \cite[0.5.B(iii)]{Gromov1989}  If $A$ is an algebraic subvariety of $\C^n$ of codimension at least 2, then $Y=\C^n\setminus A$ is as in (b).  Namely, we take the vector fields $v_1, \ldots, v_k$ to be of the form $v(z)=f(\pi(z))b$, where:
\begin{itemize}
\item  $b\in\C^n\setminus\{0\}$,
\item  $\pi:\C^n\to\C^{n-1}$ is a linear projection with $\pi(b)=0$ such that $\pi\vert A$ is proper,
\item  $f:\C^{n-1}\to\C$ is a polynomial that vanishes on the subvariety $\pi(A)$.
\end{itemize}

\begin{exercise}  \label{x:Gromov's-example}
Show that the flow of $v$ is given by $\phi^t(z)=z+tf(\pi(z))b$.  In particular, $v$ is $\C$-complete.  The flow fixes $A$ pointwise, so it restricts to a complete flow on $Y$.  Show that there is enough freedom to choose $b$, $\pi$, and $f$ that finitely many vector fields of this form span the tangent space $T_z\C^n$ at every point $z\in Y$.   
\end{exercise}

We say that $Y$ is \textit{algebraically elliptic} because the dominating spray that we have produced is algebraic.  (It is of course unusual for the flow of a $\C$-complete algebraic vector field to be algebraic.)

This example plays a key role in the proof of Forster's conjecture.

The assumption that $A$ is algebraic cannot be removed (although it can be relaxed to $A$ being a \textit{tame}\footnote{This means that there is an automorphism $\Phi$ of $\C^n$ such that $\Phi(A)$ does not accumulate at every point of the hyperplane at infinity.} analytic subvariety \cite[Proposition 5.5.14]{Forstneric2011}).  Rosay and Rudin showed that if $n\geq 2$, then there is a closed discrete set in $\C^n$ which is unavoidable by nondegenerate holomorphic maps from $\C^n$ to $\C^n$ \cite[Theorem 4.5]{RosayRudin1988}.
  
(d)  Hypersurfaces
\[ \{(z_1,\ldots,z_n,u,v)\in\C^{n+2}: uv=f(z_1,\ldots,z_n)\} \]
in $\C^{n+2}$, where $f\in\cO(\C^n)$ has $df\neq 0$ at every point of $f^{-1}(0)$, are are as in (b) \cite[Theorem 2]{KalimanKutzschebauch2008}.  These hypersurfaces are of interest because of their relevance to some deep open questions \cite[\S 4]{KalimanKutzschebauch2008}.

Clearly, the examples in (c) and (d) are only homogeneous in exceptional cases.
\end{example}

\begin{exercise}  \label{x:Riemann-elliptic}
Show that a Riemann surface is elliptic if and only if it is not covered by~$\D$.
\end{exercise}

Known ways to produce new elliptic manifolds from old are very limited.

\begin{proposition}  \label{p:new-elliptic-from-old}
(a)  If $Y_1$ and $Y_2$ are elliptic, then so is $Y_1\times Y_2$.

(b)  If $X \to Y$ is an unbranched holomorphic covering map and $Y$ is elliptic, then $X$ is elliptic.
\end{proposition}

\begin{proof}
(a)  Let $\pi_j:Y_1\times Y_2\to Y_j$, $j=1,2$, be the projections.  If $s_j:E_j\to Y_j$ is a dominating spray, then 
\[ \pi_1^*E_1\oplus \pi_2^*E_2\to Y_1\times Y_2, \quad (v_1,v_2)\mapsto (s_1(v_1), s_2(v_2)), \] 
is a dominating spray.

(b)  Let $p:X\to Y$ be a covering map and $s:E\to Y$ be a dominating spray.  Let $p^*E\to X$ be the pullback by $p$ of the spray bundle $E\to Y$.  We obtain a dominating spray on $X$ with spray bundle $p^*E$ as a lifting in the following square.
\[ \xymatrix{
X \ar@{=}[r] \ar@{^{(}->}[d] & X \ar[d]^p \\ p^*E \ar@{-->}[ur] \ar[r] & Y
} \]
The left vertical map is the inclusion of the zero section $x\mapsto (x, 0_{p(x)})$.  The bottom map is $(x,v)\mapsto s(v)$.
\end{proof}

There are weaker properties called \textit{subellipticity} and \textit{weak subellipticity}.  Subellipticity of $Y$ requires finitely many sprays that together dominate at each point of $Y$.  Weak subellipticity requires countably many sprays that together dominate at each point.  There are many examples of subelliptic and weakly subelliptic manifolds that are not known to be elliptic \cite[\S 5.5, \S 6.4]{Forstneric2011}.

\subsection{Ellipticity implies CAP}  We will now explain the workings of Gromov's linearisation method.  The idea is to represent maps by sections of vector bundles and use the Oka-Weil approximation theorem for such sections.\footnote{That ellipticity implies CAP is usually proved by showing that an elliptic manifold satisfies the \textit{homotopy Runge property} (see e.g.\ \cite[Theorem 2.22]{Forstneric2013}).  We present a variant of this property, because we find the variant attractive and also to offer something not contained in other references.}

Let $\Omega$ be a Stein domain in a Stein manifold $X$.  Let $\rho_\Omega^X:\cO(X)\to\cO(\Omega)$ be the restriction map.  We say that $\Omega$ is \textit{Runge} if the image of $\rho_\Omega^X$ is dense in $\cO(\Omega)$ (with respect to the compact-open topology).  Equivalently, $\Omega$ can be exhausted by compact subsets that are $\cO(X)$-convex (and not merely $\cO(\Omega)$-convex).  For example, if $\Omega$ is a sublevel set of a smooth strictly plurisubharmonic exhaustion of $X$, then $\Omega$ is Runge.

\begin{theorem}  \label{t:ell-implies-Runge}
Let $Y$ be an elliptic manifold and let $\Omega$ be a Runge domain in a Stein manifold $X$.  Let $\rho_\Omega^X:\cO(X,Y)\to\cO(\Omega,Y)$ be the restriction map.  The closure of the image of $\rho_\Omega^X$ is the union of some of the path components of $\cO(\Omega,Y)$.
\end{theorem}

Loosely speaking, in the setting of the theorem, approximability is deformation-invariant.

\begin{corollary}  \label{t:ell-implies-Oka}
An elliptic manifold is Oka.
\end{corollary}

\begin{proof}
Suppose $Y$ is elliptic.  Let $K$ be a convex compact subset of $\C^m$, and $f:U\to Y$ be a holomorphic map defined on a neighbourhood $U$ of $K$.  We may assume that $U$ is convex, so $U$ is a Runge domain in $\C^m$, and that $0\in K$.  Considering the path $t\mapsto f_t:U\to Y$, where $f_t(z)=f(tz)$, we see that $f=f_1$ lies in the same path component of $\cO(U,Y)$ as the constant map $f_0$.  Thus, by Theorem \ref{t:ell-implies-Runge}, $f$ can be uniformly approximated on $K$ by holomorphic maps $\C^m\to Y$.
\end{proof}

More generally, Theorem \ref{t:ell-implies-Runge} shows that $\rho_\Omega^X\cO(X,Y)$ is dense in $\cO(\Omega,Y)$ if $\cO(\Omega,Y)$ is path connected.  If either $\Omega$ or $Y$ is \textit{holomorphically contractible}, meaning that the identity map can be deformed through holomorphic maps to a constant map, then $\cO(\Omega,Y)$ is path connected.

\begin{proof}[Proof of Theorem \ref{t:ell-implies-Runge}]
Let $\pi:E\to Y$ be a vector bundle with a dominating spray $s:E\to Y$.  Let $X$ and $\Omega$ be as above.  Let $Z=X\times Y$ and let $p:Z\to X$ and $q:Z\to Y$ be the projections.  Maps $X\to Y$ correspond to sections of $p$.  We pull $E$ back to a bundle $q^*\pi:q^*E\to Z$, and obtain a map $\sigma:q^*E\to Z$, $((x,y),v)\mapsto (x,s(v))$, where $v\in E_y$, called a \textit{fibre-dominating spray}.  The word \textit{fibre} refers to the fibres of $p$.  The map $\sigma$ takes the fibre of $q^*E$ over $(x,y)\in Z$ nondegenerately into the fibre $p^{-1}(x)$ in $Z$.

The domination property of $s$ means that $Ds:\Ker D\pi\vert Y\to TY$ is an epimorphism of bundles over $Y$, that is, fibrewise surjective.  Here, $\Ker D\pi\vert Y$ is the vertical subbundle of the tangent bundle of $E$, restricted to its zero section, which we identify with $Y$, so $\Ker D\pi\vert Y$ is naturally identified with $E$ itself.  Therefore, $D\sigma:\Ker D q^*\pi\vert Z \to \Ker Dp$ is an epimorphism of bundles over $Z$.

Let $f\in\cO(S,Y)$, where $S$ is a Stein domain in $X$, and let $h:x\mapsto (x,f(x))$ be the corresponding section of $p$ over $S$.  By Siu's Stein neighbourhood theorem \cite[Theorem IX.2.13]{Demailly2012}, the Stein submanifold $h(S)$ of $S\times Y$ has a Stein neighbourhood $W$ in $S\times Y$.  Over $W$, viewed as a Stein open subset of $Z$, there is a holomorphic subbundle $F$ of $\Ker D q^*\pi\vert Z$ such that $\Ker D q^*\pi\vert Z=\Ker D\sigma\oplus F$ and $D\sigma:F\to\Ker Dp$ is an isomorphism (Proposition \ref{p:vector-bundles}(a)).  Since $\Ker D q^*\pi\vert Z$ is naturally identified with $q^*E$, we may view $F$ as a subbundle of $q^*E$.  By the inverse function theorem, $\sigma$ maps a neighbourhood of the zero section in $F\vert h(S)$ biholomorphically onto a neighbourhood of $h(S)$ in $Z$.

Now let $I \to\cO(\Omega,Y)$, $t\mapsto f_t$, be a continuous path such that $f_0\in\overline{\rho_\Omega^X\cO(X,Y)}$.  Let $h_t:x\mapsto (x,f_t(x))$ be the corresponding sections of $p$.  Let $K\subset\Omega$ be compact.  We need to show that $f_1$ can be uniformly approximated on $K$ by holomorphic maps $X\to Y$.

A Runge domain in a Stein manifold can be exhausted by relatively compact Runge subdomains (for example by analytic polyhedra).  Choose a Runge domain $U_0$ with $K\subset U_0\Subset\Omega$.  There is a partition $0=t_0<t_1<\cdots<t_k=1$ of $I$ such that for each $j=0,\ldots,k-1$ and each $t\in[t_j, t_{j+1}]$, $h_t(U_0)$ lies in a neighbourhood of $h_{t_j}(U_0)$ that is a biholomorphic image by $\sigma$ of a neighbourhood of the zero section in $F_j\vert h_{t_j}(U_0)$.  Here, $F_j$ is the bundle $F$ obtained as above with $S=\Omega$ and $h=h_{t_j}$.  Hence $h_t\vert U_0$ lifts to a holomorphic section $\xi$ of $F_j\vert h_{t_j}(U_0)$ with $\sigma\circ\xi\circ h_{t_j}=h_t$ on $U_0$.

Choose Runge domains $U_1,\ldots,U_k$ and $V_0,\ldots,V_{k-1}$ such that
\[ K\subset U_k\Subset V_{k-1}\Subset U_{k-1}\Subset\cdots\Subset U_1\Subset V_0\Subset U_0\Subset \Omega. \]
By assumption, $h_0$ can be uniformly approximated on $U_0$ by a holomorphic section $g_0$ of $p$ defined on all of $X$.  If the approximation is close enough, then $F_0$ is defined on $g_0(U_0)$ and $h_{t_1}\vert U_0$ lifts to a section $\xi$ of $F_0\vert g_0(U_0)$ with $\sigma\circ\xi\circ g_0=h_{t_1}$ on $U_0$.  By Oka-Weil for splitting morphisms (see the remarks following the proof of Proposition \ref{p:vector-bundles}), there is a bundle $F$ obtained as above with $S=X$ and $h=g_0$ that approximates $F_0$ closely enough on $g_0(V_0)$ that $h_{t_1}\vert V_0$ lifts to a section $\xi$ of $F\vert g_0(V_0)$ with $\sigma\circ\xi\circ g_0=h_{t_1}$ on $V_0$.  By Oka-Weil for sections of $F\vert g_0(X)$, $\xi$ can be uniformly approximated on $g_0(U_1)$ by global sections of $F\vert g_0(X)$.  Thus, $h_{t_1}$ can be uniformly approximated on $U_1$ by global sections of $p$.

Continuing in this way, we see that $h_1$ can be uniformly approximated on $U_k$ by global sections of $p$.  Hence, $f_1$ can be uniformly approximated on $K$ by holomorphic maps $X\to Y$.
\end{proof}

Theorem \ref{t:ell-implies-Runge} also holds for weakly subelliptic manifolds, but a more involved proof, using Gromov's technique of \textit{composed sprays}, is required.  Thus, weakly subelliptic manifolds are Oka \cite[Corollary 5.5.12]{Forstneric2011}.

%  FORSTNERIC'S THEOREM

\section{Forstneri\v c's theorem}
\label{s:forstneric}

\noindent
Forstneri\v c's proof that CAP implies not only BOPAI but also a stronger property called POPAI (see Section \ref{s:oka-manifolds}) appeared in three papers in 2005--2009 (\cite{Forstneric2005}, \cite{Forstneric2006}, \cite{Forstneric2009}).  The proof is presented in detail in \cite[Chapter 5]{Forstneric2011}.  There is a brief summary of the proof in the survey \cite{ForstnericLarusson2011} and a more detailed overview in the survey \cite{Forstneric2013}.  Here we will highlight two key ingredients in the proof, \textit{exhaustion by convex bumps} and \textit{gluing of thick sections}.  These key ingredients are already present in the proof that CAP implies AP.

Let $Y$ be a manifold satisfying CAP.  Let $X$ be a Stein manifold and $f:X\to Y$ be a continuous map which is holomorphic on a neighbourhood $U$ of a holomorphically convex compact subset $K$ of $X$.  We want to show that $f$ can be uniformly approximated on $K$ by holomorphic maps $X\to Y$.

An approximant is constructed by modifying $f$ step by step.  At each step, we make $f$ holomorphic on a slightly larger set, obtained by adding a \lq\lq bump\rq\rq\ to the set where $f$ was already holomorphic, keeping $f$ almost unchanged on the latter set.  The bumps gradually fill out the sublevel sets of a smooth strictly plurisubharmonic exhaustion $\psi:X\to\R$ with $\psi<0$ on $K$ and $\psi>0$ on $X\setminus U$.  Away from the critical points of $\psi$, the bumps have a very simple geometric shape, as we will explain in a moment.  Near a critical point, the geometry is more complicated and the procedure is different.  This is where topological obstructions could arise, but they are ruled out by $f$ being defined on all of $X$ as a continuous map.  We will not discuss the critical case (see \cite[\S 3.5]{Forstneric2013}).

At the start of each step, away from the critical points of $\psi$, $f$ is holomorphic on the closure $A$ of a relatively compact strictly pseudoconvex domain in $X$.  Near each of its boundary points, $A$ is defined by the inequality $\rho\leq 0$, where $\rho$ is a smooth strictly plurisubharmonic function.  Then $A$ with the next bump $B$ added to it is defined by the inequality $\rho\leq\epsilon\chi$, where $\chi$ is a smooth bump function with small support, and $\epsilon>0$ is small enough that $\rho-\epsilon\chi$ is strictly plurisubharmonic.  It may be arranged that $\overline{A\setminus B}\cap \overline{B\setminus A}=\varnothing$.  (Think of $A$ as a large rectangle and $B$ as a small rectangle such that $A\cap B$ and $A\cup B$ are also rectangles.)

The following lemma shows that there are holomorphic coordinates on a neighbourhood $U$ of $B$ in which $B$, $A\cap B$, and $(A\cup B)\cap U$ are strictly convex.  This is the reason why approximation on convex sets (CAP) is enough to give approximation on arbitrary holomorphically convex sets (AP).

\begin{lemma}[Narasimhan's lemma]  \label{l:Narasimhan}
Let $u$ be a smooth strictly plurisubharmonic function on a neighbourhood of the origin $0$ in $\C^n$.  If $0$ is not a critical point of $u$, then there are holomorphic coordinates on a neighbourhood of $0$ in which $u$ is strictly convex.
\end{lemma}

\begin{proof}
Take $n=1$; the proof of the general case is no more difficult.  The second-order Taylor expansion of $u$ at $0$ is
\[ \begin{aligned}
u(w) &= u(0) + u_z(0)w +\tfrac 1 2 u_{zz}(0)w^2 + u_{\bar z}(0)\bar w +\tfrac 1 2 u_{\bar z\bar z}(0)\bar w^2 + u_{z\bar z}(0)w\bar w+o(\lvert w\rvert^2) \\ &=  u(0) + 2\Re\big[u_z(0)w +\tfrac 1 2 u_{zz}(0)w^2\big] + u_{z\bar z}(0)w\bar w + o(\lvert w\rvert^2).
\end{aligned} \]
Say $u_z(0)=1$.  Take a new coordinate $w'=w+\tfrac 1 2u_{zz}(0)w^2$ near $0$.  In this new coordinate,
\[ u(w') = u(0) + 2\Re w' + u_{z\bar z}(0)w'\bar w' + o(\lvert w'\rvert^2) \]
is strictly convex near $0$.
\end{proof}

It remains to explain how holomorphicity of a map is extended from $A$ to the convex bump $B$.  We need to show that a holomorphic map $f:A\to Y$ can be uniformly approximated on $A$ by holomorphic maps $A\cup B\to Y$.  This is done in three steps.  As in the proof of Theorem \ref{t:ell-implies-Runge}, it is now convenient to view maps $X\to Y$ as sections of the projection $p:Z=X\times Y\to X$.

\noindent
\textit{Step 1:  Thickening.}  Let $V\Subset V_0$ be Stein neighbourhoods of $A$ such that $f$ is defined on $V_0$.  The Stein submanifold $f(V_0)$ of $V_0\times Y$ has a Stein neighbourhood $\Omega$ in $Z$.  By Proposition \ref{p:vector-bundles}, there are holomorphic vector fields $v_1,\ldots,v_k$ on $\Omega$ that are tangent to the fibres of $p$ and span the vertical subbundle $\Ker Dp$ of the tangent bundle $TZ$ at every point of $\Omega$.  Let $\phi_j^t$ be the flow of $v_j$.  Define a holomorphic map $F:V\times W\to Z$, where $W$ is a small enough neighbourhood of the origin in $\C^k$, by the formula
\[ F(x,w_1,\dots,w_k)=\phi_k^{w_k}\circ\cdots\circ\phi_1^{w_1}\circ f(x). \]
For every $w\in W$, $F(\cdot,w)$ is a section of $p$.  We call $F$ a \textit{thick section} of $p$ over $V$ or a \textit{holomorphic spray of sections} of $p$ over $V$ with \textit{core section} $F(\cdot, 0)=f$.  Since 
\[ \partial_{w_j} F(x,w)\big\vert_{w=0} = v_j(f(x)), \] 
the derivative 
\[ \partial_w F(x,w)\big\vert_{w=0}:\C^k\to(\Ker Dp)_{f(x)} \]
is surjective for all $x\in V$, so we say that $F$ is \textit{dominating over} $V$ (or simply \textit{dominating}).  

\noindent
\textit{Step 2:  Approximation.}  Let $D$ be a closed ball in $W$ centred at $0$.  Write $C=A\cap B$.  Since $Y$ satisfies CAP, we can uniformly approximate $F$ on $C\times D$ by a holomorphic thick section $G:B\times D\to Z$ of $p$.  This is the one place in Forstneri\v c's proof where CAP is invoked.  (In fact, the approximation needs to be done on a neighbourhood of $C\times D$, but we ignore such details here and in the following.)

\noindent
\textit{Step 3:  Gluing.}  Next we \lq\lq glue\rq\rq\ $F$ and $G$ together into a holomorphic thick section over $A\cup B$, whose core section will be our desired approximant.

The kernel of the epimorphism
\[ \partial_w F(\cdot,w)\big\vert_{w=0}:C\times\C^k\to\Ker Dp\vert f(C) \]
has a complement $E$ in the trivial bundle $C\times\C^k$ over $C$, such that
\[ \partial_w F(\cdot,w)\big\vert_{w=0}:E\to\Ker Dp\vert f(C) \]
is an isomorphism.  Hence, $F$ maps a neighbourhood of the zero section of $E$ biholomorphically onto a neighbourhood of $f(C)$ in $Z$.  Therefore, if the approximation in Step 2 is good enough, and after shrinking $W$, we obtain a holomorphic map $\gamma:C\times W\to C\times\C^k$ of the form $\gamma(x,w)=(x,\gamma_2(x,w))$, close to the inclusion, such that $F\circ\gamma=G$.

The closer $G$ is to $F$ on $C\times D$, the closer $\gamma$ will be to the inclusion.  If $\gamma$ is close enough to the inclusion and we shrink $W$ again, Forstneri\v c's \textit{splitting lemma} provides holomorphic maps $\alpha:A\times W\to A\times\C^k$ and $\beta:B\times W\to B\times \C^k$ of the same form as $\gamma$, close to the respective inclusions, such that $\gamma\circ\beta=\alpha$ on $C\times W$.  Then $F\circ\alpha=G\circ\beta$ on $C\times W$, so $F\circ\alpha$ and $G\circ\beta$ define a thick section over $A\cup B$, whose core section is the desired approximant.

The splitting lemma first appeared in \cite[\S 4]{Forstneric2003}.  It may be viewed as a nonlinear analogue of Cartan's classical \textit{Heftungslemma}.  It is a powerful tool.  In addition to the application described here, it has been used to construct proper holomorphic maps from complex curves to certain complex spaces \cite{DrinovecForstneric2007}, to extend the Poletsky theory of disc functionals from manifolds to singular spaces \cite{DrinovecForstneric2012}, and to expose boundary points of strictly pseudoconvex domains \cite{DiederichEtAl2013}.

We conclude this section with a rough sketch of the proof of the splitting lemma, following \cite[\S 5.8]{Forstneric2011}.  Let $\mathscr C$ be the Banach space of holomorphic maps $C\times W\to\C^k$ with finite supremum norm.  Similarly define the Banach spaces $\mathscr A$ and $\mathscr B$ for $A$ and $B$, respectively.

First we seek bounded linear operators $L:\mathscr C\to\mathscr A$ and $M:\mathscr C\to\mathscr B$ such that $c=L c-M c$ for all $c\in\mathscr C$.  Since $\overline{A\setminus B}\cap \overline{B\setminus A}=\varnothing$, there is a smooth function $\chi:X\to[0,1]$ such that $\chi=0$ near $\overline{A\setminus B}$ and $\chi=1$ near $\overline{B\setminus A}$.  For every $c\in\mathscr C$, $\chi(x) c(x,w)$ extends to a continuous map on $A\times W$ that vanishes on $\overline{A\setminus B}\times W$, and $(\chi(x)-1)c(x,w)$ extends to a continuous map on $B\times W$ that vanishes on $\overline{B\setminus A}\times W$.  Also, $\bar\partial(\chi c)=\bar\partial((\chi-1)c)=c\bar\partial\chi$ is a smooth $(0,1)$-form on $(A\cup B)\times W$, supported in $C\times W$, and depending holomorphically on $w\in W$.  Using a bounded linear integral operator $T$ to solve the $\bar\partial$-equation on $A\cup B$, we define
\[ \begin{aligned}
Lc(x,w) &= \chi(x)c(x,w)-T(c(\cdot,w)\bar\partial\chi)(x), \\
Mc(x,w) &= (\chi(x)-1)c(x,w)-T(c(\cdot,w)\bar\partial\chi)(x).
\end{aligned} \]
It is easily verified that $L$ and $M$ have the desired properties.

Now let $\mathscr E$ be the Banach space of holomorphic maps $\eta:C\times W\to\C^k$ such that both $\eta$ and $\partial_w\eta$ have finite supremum norm.  For $\eta\in\mathscr E$ near $\eta_0:(x,w)\mapsto w$ and $c\in\mathscr C$ near $0$, let
\[ \Phi(\eta,c)(x,w)=w+Lc(x,w)-\eta(x,w+Mc(x,w)) \]
for $x\in C$ and $w\in W$.  This defines a smooth map $\Phi$ from a neighbourhood of $(\eta_0,0)$ in the Banach space $\mathscr E\times\mathscr C$ to the Banach space $\mathscr C$.  For all $c\in\mathscr C$,
\[ \Phi(\eta_0,c)=Lc-Mc=c, \]
so $\partial_c\Phi(\eta_0,0)=\id_{\mathscr C}$.  By the implicit function theorem for Banach spaces, there is a smooth map $h$ to $\mathscr C$ from a neighbourhood of $\eta_0$ in $\mathscr E$, such that $h(\eta_0)=0$ and $\Phi(\eta,h(\eta))=0$ for all $\eta$ in this neighbourhood.  Finally, let
\[ \begin{aligned}
\alpha(x,w)&=(x, w+Lh(\gamma_2)(x,w)), \\
\beta(x,w)&=(x, w+Mh(\gamma_2)(x,w)).
\end{aligned} \]
Then $\Phi(\gamma_2,h(\gamma_2))=0$ means that $\gamma\circ\beta=\alpha$.

%  OKA MANIFOLDS

\section{Oka manifolds}
\label{s:oka-manifolds}

\noindent
Recall that a complex manifold is defined to be Oka if it satisfies the equivalent properties CAP, CIP, AP, IP, and BOPAI.  In Section \ref{s:linearisation}, we proved that elliptic manifolds are Oka and we saw some examples of elliptic manifolds.

\subsection{Properties of Oka manifolds}  We begin with three exercises.

\begin{exercise}  \label{x:dominable}
Let $Y$ be an Oka manifold.  Prove the following.

(a)  $Y$ is \textit{strongly dominable}, meaning that for every $y\in Y$, there is a holomorphic map $f:\C^{\dim Y}\to Y$ such that $f(0)=y$ and $f$ is a local biholomorphism at $0$.  (If such a map exists for some $y\in Y$, then $Y$ is called \textit{dominable}.)

(b)  $Y$ is $\C$-\textit{connected}, meaning that any two points of $Y$ can be joined by a holomorphic image of $\C$.

(c)  The Kobayashi pseudodistance on $Y$ vanishes identically.\footnote{The Kobayashi pseudodistance on $Y$ is the largest pseudodistance $d$ on $Y$ such that $d(f(z),f(w))\leq\delta(z,w)$ for all holomorphic maps $f:\D\to X$, where $\delta$ denotes the Poincar\'e distance on $\D$.  It is a nontrivial but rather easy fact that there is a \textit{largest} such pseudodistance.}

(d)  If a plurisubharmonic function on $Y$ is bounded above, then it is constant.  (This holds for $Y=\C^n$ because if $u$ is plurisubharmonic on $\C^n$, then the function $r\mapsto\sup\limits_{\lvert z\rvert=e^r} u(z)$ is convex and increasing.)
\end{exercise}

\begin{exercise}  \label{x:Riemann-Oka}
Show that a Riemann surface is Oka if and only if it is not covered by~$\D$.
\end{exercise}

\begin{exercise}  \label{x:Rosay-Rudin}
In Section \ref{s:linearisation}, we mentioned the result of Rosay and Rudin that if $n\geq 2$, then there is a closed discrete set $D\subset\C^n$ which is unavoidable by nondegenerate holomorphic maps $\C^n\to\C^n$.  Thus $\C^n\setminus D$ is not dominable and hence not Oka.  Show that $\C^n\setminus D$ is $\C$-connected.
\end{exercise}

Recall that an Oka manifold $Y$ satisfies BOP, meaning that every continuous map $f$ from a Stein manifold $X$ to $Y$ can be deformed to a holomorphic map.  In other words, the inclusion $\cO(X,Y) \hookrightarrow \cC(X,Y)$ is a surjection on path components.  (These spaces are always endowed with the compact-open topology.  Note that $\cO(X,Y)$ is closed in $\cC(X,Y)$.)  

More is true.  If $f$ belongs to a family of maps depending continuously on a parameter in a \lq\lq nice\rq\rq\ parameter space $P$, then the deformation can be made to depend continuously on the parameter.  Furthermore, if the maps parametrised by a closed subspace $Q$ of $P$ are holomorphic to begin with, then they can be left fixed during the deformation.  More precisely, $Y$ satisfies the \textit{parametric Oka property} (POP), defined as follows.

Let $X$ be a Stein manifold and $Q\subset P$ be compact subsets of $\R^n$.  For every continuous map $f:X\times P\to Y$ such that $f(\cdot,q):X\to Y$ is holomorphic for every $q\in Q$, there is a continuous map $F:X\times P\times I\to Y$ such that:
\begin{itemize}
\item  $F(\cdot,\cdot,0)=f$,
\item  $F(\cdot,q,t)=f(\cdot,q)$ for every $q\in Q$ and $t\in I$,
\item  $F(\cdot,p,1)$ is holomorphic for every $p\in P$.
\end{itemize}
As mentioned earlier, Forstneri\v c has shown that Oka manifolds satisfy a stronger property, a parametric version of BOPAI called the \textit{parametric Oka property with approximation and interpolation} (POPAI).

\begin{exercise}  \label{x:weak-equivalence}
(a)  Use POP to show that if $X$ is Stein and $Y$ is Oka, and holomorphic maps $f, g:X\to Y$ are homotopic through continuous maps, then $f$ and $g$ are homotopic through holomorphic maps.

(b)  Apply POP with the parameter space pairs $\varnothing\hookrightarrow\ast$, $\{0,1\}\hookrightarrow [0,1]$, $\ast\hookrightarrow S^k$, and $S^k\hookrightarrow B^{k+1}$ for $k\geq 1$, where $B^{k+1}$ denotes the closed $(k+1)$-dimensional ball and $S^k=\partial B^{k+1}$ is the $k$-dimensional sphere, to show that if $X$ is Stein and $Y$ is Oka, then the inclusion $\cO(X,Y)\hookrightarrow\cC(X,Y)$ is a weak homotopy equivalence with respect to the compact-open topology, that is, it induces a bijection of path components and isomorphisms of all homotopy groups.
\end{exercise}

It is natural to ask whether $\cO(X,Y)\hookrightarrow\cC(X,Y)$ is a homotopy equivalence.  The usual way to show that a weak equivalence is a homotopy equivalence is to apply Whitehead's theorem \cite[\S 10.3]{May1999}, which requires the source and the target to be CW complexes.  In our case, $\cO(X,Y)$ and $\cC(X,Y)$ carry CW structures only in trivial cases: they are metrisable, but a metrisable CW complex is locally compact.  Using more advanced topological methods, it was recently shown that if $X$ is affine algebraic, then $\cO(X,Y)\hookrightarrow\cC(X,Y)$ is a homotopy equivalence.  Moreover, $\cO(X,Y)$ is a deformation retract of $\cC(X,Y)$ \cite{Larusson2013}.  We will discuss this in more detail in Section \ref{s:space}.

\subsection{New Oka manifolds from old}  In Proposition \ref{p:new-elliptic-from-old}, we noted the limited known ways to produce new elliptic manifolds from old.  There are many more ways to construct new Oka manifolds from old.  We start with the following result.

\begin{proposition}  \label{p:down}
Let $p:Y\to Z$ be a holomorphic covering map.  Then $Y$ is Oka if and only if $Z$ is Oka.
\end{proposition}

\begin{proof}
Let us formulate the Oka property as CIP.  Let $S$ be a contractible subvariety of $\C^m$.  Suppose $Z$ is Oka and let $f:S\to Y$ be holomorphic.  Then $p\circ f:S\to Z$ extends to a holomorphic map $h:\C^m\to Z$.  Since $S\hookrightarrow\C^m$ is acyclic and $p$ is a covering map, we obtain a holomorphic lifting $\C^m\to Y$ in the square below, extending~$f$.
\[ \xymatrix{
S \ar[r]^f \ar@{^{(}->}[d] & Y \ar[d]^p \\ \C^m \ar@{-->}[ur] \ar[r]^h & Z
} \]

Now suppose $Y$ is Oka and let $f:S\to Z$ be holomorphic.  Since $S$ is contractible and $p$ is a covering map, $f$ lifts to a holomorphic map $g:S\to Y$.  Then $g$ extends to a holomorphic map $h:\C^m\to Y$, and $p\circ h$ extends $f$.
\end{proof}

The proposition provides examples of Oka manifolds that are not known to be elliptic.

\begin{example}  \label{e:Oka-not-elliptic}
(a)  A \textit{Hopf manifold} is a compact manifold with universal covering space $\C^n\setminus\{0\}$, $n\geq 2$.  By Example \ref{e:elliptic}(c), $\C^n\setminus\{0\}$ is elliptic and hence Oka.  Thus every Hopf manifold is Oka.

(b)  Let $\Gamma$ be a lattice in $\C^n$, $n\geq 2$, and $T=\C^n/\Gamma$ be the corresponding torus.  Let $F\subset T$ be finite.  The preimage $E$ of $F$ in $\C^n$ is the union of finitely many translates of~$\Gamma$.  It may be shown that $E$ is tame \cite[Proposition 4.1]{BuzzardLu2000}, so the universal covering $\C^n\setminus E$ of $T\setminus F$ is elliptic \cite[Proposition 5.5.14]{Forstneric2011}.  Hence, $T\setminus F$ is Oka.
\end{example}

\begin{exercise}  \label{x:product-union-retract}
(a)  Show that the product of two Oka manifolds is an Oka manifold.

(b)  Show that if a complex manifold $Y$ is exhausted by Oka domains $\Omega_1\subset\Omega_2\subset\Omega_3\subset\cdots$, then $Y$ is Oka.

This property was used to prove that minimal Enoki surfaces are Oka (see Section \ref{s:deformations}).  It also shows that a long $\C^n$ (a manifold that is an increasing union of domains that are biholomorphic to $\C^n$) is Oka.  A long $\C^2$ need not be Stein \cite{Wold2010}.

(c)  Show that a retract of an Oka manifold is Oka.  More explicitly, if $Z$ is a submanifold of an Oka manifold $Y$ and there is a holomorphic map $r:Y\to Z$ such that $r\circ i=\id_Z$, where $i:Z\hookrightarrow Y$ is the inclusion, then $Z$ is Oka.
\end{exercise}

The most powerful method for constructing new Oka manifolds from old is a generalisation of Proposition \ref{p:down}, in which covering maps are replaced by a much larger class of maps, called \textit{Oka maps}.  What should it mean for a holomorphic map to be Oka?  First, it means that the map satisfies POPAI suitably formulated for maps, so that $Y\to\ast$ satisfies POPAI as a map if and only if $Y$ does as a manifold.  As above, in the case of manifolds, let us formulate the simpler POP (without approximation and interpolation).

A holomorphic map $\pi:Y\to Z$ satisfies POP if whenever:
\begin{itemize}
\item $X$ is a Stein manifold, 
\item $Q\subset P$ are compact subsets of $\R^n$, 
\item $f:X\times P\to Z$ is continuous and $f(\cdot,p)$ is holomorphic for all $p\in P$,
\item $g_0:X\times P\to Y$ is a continuous lifting of $f$ by $\pi$ (so $\pi\circ g_0=f$) and $g_0(\cdot,q)$ is holomorphic for all $q\in Q$,
\end{itemize}
there is a continuous deformation $g_t:X\times P\to Y$ of $g_0$ such that for all $t\in I$,
\begin{itemize}
\item $\pi\circ g_t=f$, 
\item $g_t=g_0$ on $X\times Q$,
\item $g_1(\cdot,p)$ is holomorphic for all $p\in P$.
\end{itemize}
\[ \xymatrix{
X\times Q \ar[r] \ar@{^{(}->}[d] & Y \ar[d]^\pi \\ X\times P \ar[ur]^{g_{{}_t}} \ar[r]^>>>>f & Z
} \]

When $P$ is a singleton and $Q$ is empty, POP simply says that if $X$ is Stein and $f:X\to Z$ is holomorphic, then every continuous lifting $X\to Y$ of $f$ by $\pi$ can be deformed through such liftings to a holomorphic lifting.  In particular, if $Z$ is Stein, every continuous section of $\pi$ can be deformed to a holomorphic section.

The class of holomorphic maps satisfying POPAI is not closed under composition, as demonstrated by the simple example $\D\hookrightarrow\C\to\ast$.  Homotopy-theoretic considerations show that it is natural and gives a much better-behaved property to define a holomorphic map to be Oka if it satisfies POPAI and is a \textit{topological fibration}.\footnote{Oka maps were first defined in \cite[\S 16]{Larusson2004}.  See also \cite{Forstneric2010}.}   

By a topological fibration, we mean a Serre fibration or a Hurewicz fibration \cite[Chapter 7]{May1999}: in our case, the two notions are equivalent \cite[Theorem 5.1]{Arnold1972}.  A continuous map $Y\to Z$ is a Hurewicz fibration if it satisfies the \textit{covering homotopy property}, meaning that every commuting square of continuous maps
\[ \xymatrix{
A \ar[r] \ar[d]_{i_0} & Y \ar[d] \\ A\times I \ar[r] \ar@{-->}[ur] & Z
} \]
where $A$ is any topological space and $i_0(a)=(a,0)$, has a continuous lifting.  The notion of a Serre fibration restricts $A$ to be a cube, or a polyhedron, or a CW complex: these three choices all define the same property.  The most important example of a topological fibration is a fibre bundle, that is, a continuous map that is locally trivial over the base (at least if the base is paracompact).

\begin{exercise}  \label{x:oka-maps}
Prove the following.  Feel free to work with POP instead of POPAI, that is, to take an Oka map to be a holomorphic map that is a topological fibration and satisfies POP.  You might even take $P$ to be a singleton and $Q$ to be empty.

(a)  A manifold $Y$ is Oka if and only if the map from $Y$ to a point is Oka.

(b)  A holomorphic covering map is Oka.

(c)  The class of Oka maps is closed under composition.  

(d)  The pullback of an Oka map by an arbitrary holomorphic map is Oka.  

(e)  A retract of an Oka map is Oka.

(f)  An Oka map is a submersion.  (Here you need to formulate and apply the basic Oka property with interpolation for an Oka map, with the Stein manifold $X$ being a ball and the subvariety of $X$ being its centre.)

(g)  The connected components of the fibres of an Oka map are Oka manifolds.

(h)  A globally trivial bundle $Z\times F\to Z$ is Oka if and only if the fibre $F$ is Oka.
\end{exercise}

The theory of Oka manifolds and Oka maps fits into an abstract homotopy-theoretic framework in a rigorous way \cite[Appendix]{Forstneric2013}.  Properties (c), (d), and (e) of Oka maps reflect the fact that Oka maps are \textit{fibrations} in a certain homotopy-theoretic structure \cite[\S 20]{Larusson2004}.

Here is our generalisation of Proposition \ref{p:down}.

\begin{theorem}  \label{t:down}
Let $Y\to Z$ be a surjective Oka map.  Then $Y$ is Oka if and only if $Z$ is Oka.
\end{theorem}

\begin{proof}
The proof is similar to the proof of Proposition \ref{p:down} and is left as an exercise.
\end{proof}

The image of a topological fibration is a union of path components of the target.  Since in these notes we take a complex manifold to be connected by definition, the surjectivity assumption in the theorem is in fact superfluous.

Clearly, the next question is how to recognise Oka maps.  A major theorem of Forstneri\v c provides the best available answer \cite[Corollary 1.3]{Forstneric2010}.  The following is a simple version of the theorem, originating from \cite[\S 3.3.C']{Gromov1989}.  It combines a twisted version of POPAI for maps from a Stein manifold to an Oka manifold with the basic fact that a fibre bundle is a topological fibration.

\begin{theorem}  \label{t:bundles}
The projection of a holomorphic fibre bundle with Oka fibres is an Oka map.
\end{theorem}

Here are two classes of Oka manifolds obtained by using Theorems \ref{t:down} and \ref{t:bundles}.

\begin{example}  \label{e:more-Oka}
(a)  A \textit{rational surface} is a compact complex surface that is birationally equivalent to the projective plane $\P_2$.  Rational surfaces form one of the classes in the Enriques-Kodaira classification of compact complex surfaces \cite{Barth-et-al2004}.  A minimal rational surface is either $\P_2$ itself or a \textit{Hirzebruch surface}, which is a holomorphic fibre bundle over $\P_1$ with fibre $\P_1$.  Being homogeneous, $\P_n$ is elliptic and hence Oka for all $n\geq 1$.  Thus all minimal rational surfaces are Oka.

It is a fundamental problem to determine which compact complex surfaces are Oka.  For an up-to-date account of what is known, see \cite{ForstnericLarusson2012}.

(b)  \textit{Toric varieties} have a rich combinatorial structure and are the subject of very active research \cite{CoxLittleSchenck2011}.  One way to define a toric variety is to say that it is a normal algebraic variety $Y$ containing a torus $T=(\C^*)^n$ as a Zariski-open subset, such that the action of $T$ on itself extends to an algebraic action on $Y$.  I showed that every smooth toric variety $Y$ is Oka \cite[Theorem 2.17]{Forstneric2013}.  

The proof goes as follows.  By the structure theory of toric varieties and a little geometric invariant theory, $Y$ is of the form $Y_0\times(\C^*)^k$, where $Y_0$ carries a holomorphic fibre bundle $\C^m\setminus A$, whose fibre is a complex Lie group, where $A$ is a finite union of linear subspaces of $\C^m$ of codimension at least 2.
\end{example}

\begin{exercise}  \label{x:toric}
Based on this information, explain why $Y$ is Oka.
\end{exercise}

The Oka principle first appeared in Oka's 1939 result that a second Cousin problem on a domain of holomorphy can be solved by holomorphic functions if it can be solved by continuous functions.  In other words, a holomorphic line bundle over a domain of holomorphy, or more generally over a Stein manifold, is holomorphically trivial if it is topologically trivial.  In 1958, Grauert extended Oka's theorem to vector bundles (and some more general fibre bundles).  We can now establish Grauert's result as follows.  

Let $E$ be a holomorphic vector bundle of rank $n$ over a Stein manifold $X$.  The frame bundle $F$ of $E$ is a holomorphic fibre bundle over $X$ whose fibre is the complex Lie group $\mathrm{GL}(n,\C)$, so the projection $F\to X$ is an Oka map by Theorem \ref{t:bundles}.  Hence every continuous section of $F$ can be deformed to a holomorphic section.  Since $E$ is topologically trivial if and only if $F$ has a continuous section, and holomorphically trivial if and only if $F$ has a holomorphic section, it follows that $E$ is holomorphically trivial if it is topologically trivial.

\begin{example}  \label{e:eremenko}
A holomorphic submersion with Oka fibres need not be an Oka map, even if it is smoothly locally trivial.  Here is an example \cite[Example 6.6]{ForstnericLarusson2011}.   Let $g:\mathbb D\to \C$ be a smooth function.  Let $\pi:E_g=\mathbb D\times \C\setminus\Gamma_g \to \mathbb D$ be the projection, where $\Gamma_g$ denotes the graph of $g$.  Clearly, $\pi$ is smoothly trivial and each fibre $\pi^{-1}(z)\cong\C\setminus\{g(z)\} \cong \C^*$ is an Oka manifold.  However, if $\pi$ is an Oka map, then $g$ is holomorphic.  Namely, if $\pi$ is Oka, then the smooth lifting $f:\mathbb D\times\C^*\to E_g$, $(z,w)\mapsto (z,w+g(z))$, of the projection $p:\mathbb D\times\C^*\to\mathbb D$ by $\pi$ can be deformed to a holomorphic lifting $h: \mathbb D\times \C^*\to E_g$.
\[ \xymatrix{ & E_g \ar[d]^{\pi} \\ 
\D\times \C^* \ar@<1ex>[ur]^f \ar[ur]_h \ar[r]_{\ \ \ p} & \D } \]
For each $z\in\mathbb D$, $g(z)$ is the missing value in the range of the holomorphic map $h(z,\cdot):\C^*\to \C$. A deep theorem of Eremenko \cite{Eremenko2006} now implies that $g$ is holomorphic.
\end{example}

We have seen that the class of Oka manifolds is much more flexible than the class of elliptic manifolds.  Oka manifolds are our primary objects of interest.  We view ellipticity as an auxiliary property, a useful geometric sufficient condition for the Oka property to hold.  However, no Oka manifolds are known not to be elliptic.  Gromov observed that Stein Oka manifolds are elliptic \cite[\S 3.2.A]{Gromov1989}.

%  AFFINE SIMPLICES IN OKA MANIFOLDS

\section{Affine simplices in Oka manifolds}
\label{s:affine-simplices}

\noindent
The homotopy type of a CW complex, for example a manifold, or more generally the weak homotopy type of an arbitrary topological space $X$, is encoded in the \textit{singular set} $sX$ of $X$.   The singular set of $X$ consists of a sequence $sX_0, sX_1, sX_2, \ldots$ of sets, where $sX_n$ is the set of $n$-simplices in $X$, that is, the set of all continuous maps into $X$ from the standard $n$-simplex
\[T_n=\{(t_0,\ldots,t_n)\in\R^{n+1}:t_0+\cdots+t_n=1,\, t_0,\ldots,t_n\geq 0\},\]
along with face maps $d_j:sX_n\to sX_{n-1}$ and degeneracy maps $s_j:sX_n\to sX_{n+1}$ for $j=0,\ldots,n$.  A face map acts by restricting an $n$-simplex to one of the faces of $T_n$, whereas a degeneracy map acts by precomposing an $n$-simplex by a map that collapses $T_{n+1}$ onto one of its faces.  The face maps and degeneracy maps satisfy several identities, for example $d_is_j=s_{j-1}d_i$ for $i<j$.  This kind of structure is called a \textit{simplicial set}.  For an introduction to simplicial sets, see \cite{GoerssJardine1999} or \cite{May1992}.  

Given a simplicial set $S$, we can construct a CW complex $\lvert S\rvert$, called the \textit{geometric realisation} of $S$, by gluing together all the simplicies in $S$ using the face and degeneracy maps.  The realisation functor is left adjoint to the singular functor, meaning that there is a natural bijection
\[ \Hom(\lvert S\rvert, X) \cong \Hom(S, sX) \]
for every singular set $S$ and topological space $X$.  Here, on the left we have continuous maps between topological spaces, and on the right morphisms of singular sets.  It can be shown that the natural map $\lvert sX\rvert\to X$ is a weak homotopy equivalence, so it is a homotopy equivalence if $X$ is a CW complex.

The \textit{affine singular set} $eX$ of a complex manifold $X$ is the simplicial set whose $n$-simplices for each $n\geq 0$ are the holomorphic maps into $X$ from the affine $n$-simplex 
\[A_n=\{(t_0,\dots,t_n)\in\C^{n+1}:t_0+\dots+t_n=1\},\]
viewed as a complex manifold biholomorphic to $\C^n$, with the obvious face maps and degeneracy maps.   If $X$ is Brody hyperbolic (meaning that every holomorphic map $\C\to X$ is constant), then $eX$ is discrete and carries no topological information about $X$.  On the other hand, when $X$ is Oka, $eX$ is \lq\lq large\rq\rq.

A holomorphic map $A_n\to X$ is determined by its restriction to $T_n\subset A_n$, so we have a monomorphism $eX\hookrightarrow sX$ of simplical sets.  When $X$ is Oka, $eX$, which is of course much smaller than $sX$, carries the homotopy type of $X$.  More precisely, the monomorphism $eX\hookrightarrow sX$ is the inclusion of a deformation retract \cite[Theorem~1]{Larusson2009b}.  In particular, we have the following result.

\begin{theorem}  \label{t:homotopically-elliptic}
An Oka manifold $X$ is homotopy equivalent to the CW complex $\lvert eX\rvert$, which is constructed from entire maps into $X$.
\end{theorem}

Even for complex Lie groups, this result was not previously known.  The proof is short, but technical, so we refer those who are interested to \cite{Larusson2009b}.

%  THE SPACE OF HOLOMORPHIC MAPS

\section{The space of holomorphic maps from a Stein manifold to an Oka manifold}
\label{s:space}

\noindent
Recall that an Oka manifold $Y$ satisfies BOP, meaning that every continuous map $f$ from a Stein manifold $X$ to $Y$ can be deformed to a holomorphic map.  It is natural to ask whether this can be done for all $f$ at once, in a way that depends continuously on $f$ and leaves $f$ fixed if it is holomorphic to begin with.  In other words, is $\cO(X,Y)$ a deformation retract of $\cC(X,Y)$?\footnote{A subspace $B$ of a topological space $A$ is a \textit{deformation retract} of $A$ if there is a continuous map $h:A\times I\to A$ such that $h(a,0)=a$, $h(b,t)=b$, and $h(a,1)\in B$ for all $a\in A$, $b\in B$, and $t\in I$.  (The term \textit{strong} deformation retract is sometimes used, especially in the older literature.)  We call $h$ a \textit{deformation} of $A$ onto $B$.  Then $h(\cdot,1):A\to B$ is a homotopy inverse for the inclusion $B\hookrightarrow A$.  In particular, $A$ and $B$ are homotopy equivalent.}  Equivalently, we are asking whether $Y$ satisfies POP for \textit{every} pair of parameter spaces $Q\hookrightarrow P$, or simply for the universal pair $\cO(X,Y)\hookrightarrow\cC(X,Y)$.\footnote{One motivation for asking this question is the desire to clean up what is known about the pairs of parameter spaces $Q\hookrightarrow P$ for which POP holds.  As already mentioned, Oka manifolds satisfy POP when $P$ and $Q$ are compact subsets of $\R^n$.  They also satisfy POP for a subcomplex $Q$ of a CW complex $P$, as well as for more general cellular complexes \cite[Section 16]{Larusson2004}.  And elliptic manifolds satisfy POP when $P$ and $Q$ are arbitrary compact spaces \cite[Theorem 6.2.2]{Forstneric2011}.}

Our starting point is the fact that if $X$ is Stein and $Y$ is Oka, then the inclusion $\cO(X,Y)\hookrightarrow\cC(X,Y)$ is a weak homotopy equivalence (Exercise \ref{x:weak-equivalence}(b)).  We want to upgrade the inclusion, not only to a homotopy equivalence, but to the inclusion of a deformation retract.  The commonly used sufficient condition that $\cC(X,Y)$ is a CW complex and $\cO(X,Y)$ is a subcomplex of $\cC(X,Y)$ does not apply.  As already mentioned, $\cO(X,Y)$ and $\cC(X,Y)$ are metrisable, but a metrisable CW complex is locally compact.  The only applicable sufficient condition known to me is that $\cO(X,Y)$ and $\cC(X,Y)$ are \textit{absolute neighbourhood retracts} (ANRs).  For the algebraic topology behind these statements,\footnote{It includes, among older theory, the \textit{mixed model structure} on the category of topological spaces, introduced by Cole in 2006.  The best reference on ANRs is \cite{vanMill1989}.} see~\cite{Larusson2013} and the references there.

So what are absolute neighbourhood retracts?  And how can we show that $\cO(X,Y)$ and $\cC(X,Y)$ are absolute neighbourhood retracts?

Absolute neighbourhood retracts are metric spaces (or rather, metrisable spaces) with some very nice properties and an interesting relationship with CW complexes.  A finite-dimensional metric space is an ANR if and only if it is locally contractible, meaning that every neighbourhood $U$ of each point contains a neighbourhood which is contractible in $U$.  For infinite-dimensional spaces, the notion is stronger and more subtle.  A metrisable space $X$ is an ANR if it satisfies the following nontrivially equivalent conditions.
\begin{itemize}
\item  If $X$ is homeomorphically embedded as a closed subspace of a metric space $Y$, then some neighbourhood of $X$ in $Y$ retracts onto $X$.  (This explains the name ANR.)
\item  $X$ is homeomorphic to a closed subset of a convex subset $C$ of a Banach space, and $X$ has a neighbourhood in $C$ that retracts onto $X$.  (Every metric space $X$ is homeomorphic to a closed subset of a convex set in the Banach space of bounded continuous functions $X\to\R$.)
\item  For every metric space $A$ with a closed subspace $B$, every continuous map $B\to X$ can be extended to a neighbourhood of $B$ in $A$ (depending on the map).
\item  $X$ is locally contractible, and for every metric space $A$ with a closed subspace $B$, the inclusion $B\hookrightarrow A$ has the homotopy extension property with respect to continuous maps into $X$.\footnote{If $X$ was an arbitrary topological space, then $B\hookrightarrow A$ would have to be a cofibration for the homotopy extension property to hold, that is, $B$ would have to be a deformation retract of a neighbourhood in $A$.}
\item  Every open subset of $X$ has the homotopy type of a CW complex.\footnote{This remarkable characterisation is due to Cauty.  The basic theory of ANRs dates back to the mid-20th century.  Cauty's result, published in 1994, is an example of a deep result on ANRs proved more recently.  ANRs are not much in the spotlight nowadays, although they do appear in substantial recent work in geometric group theory and geometric topology.}
\item  Every open cover $\mathscr U$ of $X$ has a refinement $\mathscr V$ such that if $S$ is a simplicial complex (with the weak topology) with a subcomplex $T$ containing all the vertices of $S$, then every continuous map $\phi_0:T\to X$ such that for each simplex $\sigma$ of $S$, $\phi_0(\sigma\cap T)\subset V$ for some $V\in\mathscr V$, extends to a continuous map $\phi:S\to X$ such that for each simplex $\sigma$ of $S$, $\phi(\sigma)\subset U$ for some $U\in\mathscr U$.
\end{itemize}
The last condition is called the \textit{Dugundji-Lefschetz property}.  It is a little complicated to state, but of the six conditions, it is the one that is most likely to be possible to verify in practice.

\begin{exercise}  \label{x:dugundji}
What does the Dugundji-Lefschetz property say when $S=[0,1]$ and $T=\{0,1\}$?
\end{exercise}

Here are some further properties of ANRs.
\begin{itemize}
\item  An open subset of an ANR is an ANR.
\item  A metric space with an open cover by ANRs is an ANR.  Thus being an ANR is a local topological property.
\item  A CW complex is an ANR if and only if it is metrisable $=$ locally finite $=$ first countable $=$ locally compact.
\item  A topological space has the homotopy type of an ANR if and only if it has the homotopy type of a CW complex.\footnote{As an amusing consequence, if every open subset of a metric space $X$ has the homotopy type of an ANR, then every open subset of $X$ \textit{is} an ANR.}
\end{itemize}

Whether $\cC(X,Y)$ is an ANR is of course a topological question.  What is probably the optimal answer may be deduced from recent work of Smrekar and Yamashita \cite{SmrekarYamashita2009}.

\begin{proposition}  \label{p:C-is-ANR}
Let $X$ be a finitely dominated countable CW complex and let $Y$ be a locally finite countable CW complex.  Then $\cC(X,Y)$ is an ANR.
\end{proposition}

A CW complex $X$ is \textit{finitely dominated} if it has the homotopy type of a compact space.  Equivalently, $\id_X$ is homotopic to a map with relatively compact image.  This is strictly weaker than having the homotopy type of a finite CW complex.

The surprise in \cite{Larusson2013} is the discovery that if $X$ is a \lq\lq well-behaved\rq\rq\ Stein manifold and $Y$ is an Oka manifold, then the parametric Oka property with approximation with respect to parameter spaces that are finite polyhedra can be used to show that $\cO(X,Y)$ has the Dugundji-Lefschetz property.  The proof is direct and does not rely on $\cC(X,Y)$ being an ANR.

\begin{theorem}  \label{t:O-is-ANR}
Let $X$ be a Stein manifold with a strictly plurisubharmonic Morse exhaustion with finitely many critical points, and let $Y$ be an Oka manifold.  Then $\cO(X,Y)$ is an ANR.
\end{theorem}

The assumption on $X$ is satisfied if $X$ is affine algebraic.  Namely, algebraically embed $X$ into $\C^n$ for some $n$.  For a generic choice of $a\in\C^n$, the smooth strictly plurisubharmonic exhaustion $X\to[0,\infty)$, $z\mapsto\lVert z-a\rVert^2$, is a Morse function with finitely many critical points.  And then $X$ has the homotopy type of a finite CW complex.

\begin{corollary}  \label{c:ANR}
Let $X$ be an affine algebraic manifold and $Y$ be an Oka manifold.  Then $\cO(X,Y)$ is a deformation retract of $\cC(X,Y)$.
\end{corollary}

The assumptions on $X$ can be relaxed.  For example, $\cO(\C\setminus\mathbb N,\C^*)$ is a deformation retract of $\cC(\C\setminus\mathbb N,\C^*)$, even though $\C\setminus\mathbb N$ is not affine algebraic and not even finitely dominated.

Finally, recall that if $X$ is Stein and $Y$ is Oka, then every path in $\cC(X,Y)$ joining two points in $\cO(X,Y)$ can be deformed, keeping the end points fixed, to a path in $\cO(X,Y)$ (Exercise \ref{x:weak-equivalence}(a)).  When $X$ is affine algebraic, Corollary \ref{c:ANR} provides a \textit{local} or a \textit{controlled} version of this property.  Namely, every neighbourhood $U$ in $\cC(X,Y)$ of a point in $\cO(X,Y)$ contains a neighbourhood $V$ such that every path in $V$ joining two points in $\cO(X,Y)$ can be deformed within $U$, keeping the end points fixed, to a path in $\cO(X,Y)$.

%  DEFORMATIONS OF OKA MANIFOLDS

\section{Deformations of Oka manifolds}
\label{s:deformations}

\noindent
To determine how the Oka property behaves with respect to deformations of compact complex manifolds is a problem of fundamental importance.  Let $\pi:X\to B$ be a family of compact complex manifolds, that is, a proper holomorphic submersion, and therefore a smooth fibre bundle, from a complex manifold $X$ onto a complex manifold $B$.  Since the fibres of $\pi$ are mutually diffeomorphic, we may view $\pi$ as giving a variation of complex structure on a fixed compact smooth manifold.  Write $X_t$ for the compact complex manifold $\pi^{-1}(t)$, $t\in B$.  We would like to say as much as possible about the set $O_\pi$ of those $t\in B$ for which $X_t$ is Oka.  Here is what is known.

\begin{theorem}  \label{t:deform}
Let $\pi:X\to B$ be a family of compact complex manifolds.  Let $O_\pi$ be the set of $t\in B$ for which $\pi^{-1}(t)$ is Oka.  Then:
\begin{enumerate}
\item[(a)]  $O_\pi$ is $G_\delta$.
\item[(b)]  $O_\pi$ need not be closed.
\end{enumerate}
\end{theorem}

Part (a) is \cite[Corollary 8]{Larusson2012}.  Part (b) is \cite[Corollary 5]{ForstnericLarusson2012}.  It is an open question whether $O_\pi$ is open, that is, whether the Oka property is stable.

By comparison, the set of $t\in B$ for which $X_t$ is Kobayashi hyperbolic\footnote{A complex manifold $X$ is Kobayashi hyperbolic if there is a metric (a nondegenerate distance function) $d$ on $X$ such that $d(f(z),f(w))\leq\delta(z,w)$ for all holomorphic maps $f:\D\to X$.  Here, $\delta$ denotes the Poincar\'e distance on $\D$.  If $X$ is Kobayashi hyperbolic, then $X$ is Brody hyperbolic, meaning that every holomorphic map $\C\to X$ is constant.  The converse holds if $X$ is compact.  The theory of Kobayashi hyperbolicity is a rich subfield of complex geometry.  See \cite{Kobayashi1998}.  We can think of the Oka property as an \lq\lq anti-hyperbolicity\rq\rq\ property.  The only manifold that is both Oka and Kobayashi hyperbolic is the point.} is open \cite[Theorem 3.11.1]{Kobayashi1998}, but not necessarily closed \cite{BrodyGreen1977}.

In the remainder of this section, we sketch a proof of Theorem \ref{t:deform}.  First, let $Y$ be a compact complex manifold.  We fix a Hermitian metric $\omega$ on $Y$.  We use it to filter sets of holomorphic maps by normal families.  For our purposes, the choice of filtration is immaterial.  To get a quantitative handle on the Oka property of $Y$, we introduce the following definition.

By a \textit{quintuple} we mean a quintuple $(K, U, V, r, \epsilon)$, where $K$ is a nonempty compact subset of $\C^k$, $k\geq 1$, $U\subset V\Subset\C^k$ are neighbourhoods of $K$, $r>0$, and $\epsilon>0$.  Note that $U$ and $V$ are assumed to be relatively compact in $\C^k$.  For every quintuple $(K, U, V, r, \epsilon)$, let
\[ \sigma(K, U, V, r, \epsilon)(Y)=\sup_{\substack{f:U\to Y \textrm{ hol.} \\ \lVert f^*\omega\rVert_U\leq r}} \inf_{\substack{g:V\to Y \textrm{ hol.} \\ d(f,g)<\epsilon \textrm{ on }K}} \lVert g^*\omega \rVert_V\in[0,\infty]. \]
Here, $\lVert\cdot\rVert$ denotes the supremum norm with respect to the Euclidean metric on $\C^k$, and the distance $d(f,g)$ is with respect to $\omega$.

Clearly, $\sigma(K, U, V, r, \epsilon)(Y)$ increases as $r$ increases, $\epsilon$ decreases, $U$ shrinks, and $V$ expands.  Also, $\sigma(K, U, V, r, \epsilon)(Y)$ is finite if and only if there is $R>0$ such that every holomorphic map $f:U\to Y$ with $\lVert f^*\omega\rVert_U\leq r$ can be approximated to within $\epsilon$ on $K$ by a holomorphic map $g:V\to Y$ with $ \lVert g^*\omega \rVert_V\leq R$.  Since $Y$ is compact, whether $\sigma(K, U, V, r, \epsilon)(Y)$ is finite for all $r$ and $\epsilon$ with $K$, $U$, and $V$ fixed is independent of the choice of a Hermitian metric on $Y$.

\begin{proposition} 
\label{p:newoka}
The compact manifold $Y$ is Oka if and only if $\sigma(K, U, V, r, \epsilon)(Y)$ is finite for every quintuple $(K, U, V, r, \epsilon)$ such that $K$ is convex.
\end{proposition}

\begin{proof}  $\Leftarrow$  This is easy (and does not require compactness of $Y$).  

$\Rightarrow$  Suppose $\sigma(K, U, V, r, \epsilon)(Y)=\infty$ for some quintuple $(K, U, V, r, \epsilon)$ with $K$ convex.  This means that for every $n\in\N$, there is a holomorphic map $f_n:U\to Y$ with $\lVert f_n^*\omega\rVert_U\leq r$, such that every holomorphic map $g:V\to Y$ with $d(f_n,g)<\epsilon$ on $K$ (there may be none) has $\lVert g^*\omega \rVert_V>n$.  Since the family $\{f_n:n\in\N\}$ is equicontinuous, by passing to a subsequence, we may assume that $(f_n)$ converges locally uniformly on $U$ to a holomorphic map $f:U\to Y$.  Find $n_0$ such that $d(f,f_n)<\epsilon/2$ on $K$ for all $n\geq n_0$.  If $Y$ was Oka, we could find a holomorphic map $g:\C^k\to Y$ with $d(f,g)<\epsilon/2$ on $K$.  Then $d(f_n,g)<\epsilon$ on $K$ for $n\geq n_0$, and $\lVert g^*\omega \rVert_V\leq n$ for $n$ large enough: a contradiction.
\end{proof}

Now let $\pi:X\to B$ be a family of compact complex manifolds and set $X_t=\pi^{-1}(t)$, $t\in B$.  Take a Hermitian metric $\omega$ on $X$.  For a quintuple $(K, U, V, r, \epsilon)$, write $\sigma_{U,V}(t)$ for $\sigma(K,U,V,r,\epsilon)(X_t)$, defined using the metric $\omega\vert X_t$.

The following semicontinuity result \cite[Theorem 7]{Larusson2012} is the key to Theorem \ref{t:deform}(a).  We omit the proof.

\begin{theorem}  
\label{t:semicontinuity}
Let $(K, U, V_j, r, \epsilon)$, $j=1, 2$, be quintuples such that $V_1\Subset V_2$ and $V_2$ is Stein.  For every $t_0\in B$,
\[\limsup_{t\to t_0}\sigma_{U,V_1}(t) \leq \sigma_{U,V_2}(t_0). \]
\end{theorem}

To prove Theorem \ref{t:deform}(a), note that $\sigma$ is finite for all quintuples $(K, U, V, r, \epsilon)$ with $K$ convex if and only if $\sigma$ is finite for a suitable countable set of such quintuples.  Namely, we can take $r$ and $1/\epsilon$ to be integers, $V$ to be a ball of integer radius centred at the origin, and in between any compact convex $K$ and a neighbourhood $U$ of $K$ we can fit the convex hull of a finite set of points with rational coordinates and the interior of a larger such hull.

Fix $K, U, r, \epsilon$ and take an increasing sequence $V_1\Subset V_2\Subset \cdots \Subset \C^k$ of Stein neighbourhoods of $U$.  Write $\sigma_n$ for $\sigma_{U,V_n}$.  It suffices to show that $\bigcap\{\sigma_n<\infty\}$ is $G_\delta$.  This holds since by Theorem \ref{t:semicontinuity}, $\{\sigma_n<\infty\}$ is a neighbourhood of $\{\sigma_{n+1}<\infty\}$.  In other words, for each $n\geq 1$, there is an open set $W_n$ with
\[\{\sigma_{n+1}<\infty\}\subset W_n\subset \{\sigma_n<\infty\},\] 
so $\bigcap\{\sigma_n<\infty\}=\bigcap W_n$ is $G_\delta$.

We now turn to Theorem \ref{t:deform}(b).  The required example comes from the theory of surfaces of class VII and shows that compact complex surfaces that are Oka can degenerate to a surface that is far from being Oka.

Class VII in the Enriques-Kodaira classification comprises the non-algebraic compact complex surfaces of Kodaira dimension $\kappa=-\infty$.  They all have first Betti number $b_1=1$, so they are not K\"ahler.  Minimal surfaces of class VII fall into five mutually disjoint classes.  For second Betti number $b_2=0$, we have Hopf surfaces (see Example \ref{e:Oka-not-elliptic}(a)) and Inoue surfaces.  For $b_2\geq 1$, there are Enoki surfaces, Inoue-Hirzebruch surfaces, and intermediate surfaces; together they form the class of Kato surfaces.\footnote{The only published survey on surfaces of class VII is \cite{Nakamura1989}.  In the following, take all surfaces to be minimal.  Hopf surfaces have universal covering space $\C^2\setminus\{0\}$, so they are Oka, as already noted.  Inoue surfaces have universal covering space $\C\times\D$, so they are not Oka.  The three classes of Kato surfaces are quite small.  Each can be defined by its own rather involved construction that we shall not describe here.  A different construction of Kato produces all of them.  Blow up the open ball $B$ in $\C^2$ at a point to obtain an exceptional curve $C_1$.  Choose a point $p_1$ on $C_1$, blow up, and obtain an exceptional curve $C_2$.  Do this $n\geq 1$ times.  Choose a point $p_n$ on $C_n$ and a small open ball $D$ centred at $p_n$.  Choose a biholomorphism $f$ of a neighbourhood of $\partial B$ onto a neighbourhood of $\partial D$ that maps $\partial B$ onto $\partial D$.  Use $f$ to identify $\partial B$ and $\partial D$ in the $n$-times blown-up $B$ with $D$ removed.  This yields a Kato surface with $b_2=n$.  All Kato surfaces can be constructed in this way.  The choices made determine to which of the three classes the surface belongs.}  By the \textit{global spherical shell conjecture}, currently proved only for $b_2=1$, every minimal surface of class VII with $b_2\geq 1$ is a Kato surface. 

In \cite[Section 4.3.2]{Dloussky2013}, Dloussky constructed a family $X\to\D$ of compact complex manifolds such that the central fibre $X_0$ is an Inoue-Hirzebruch surface and the other fibres $X_t$, $t\in\D\setminus\{0\}$, are minimal Enoki surfaces.  It may shown that a minimal Enoki surface is Oka, whereas the universal covering space of an Inoue-Hirzebruch surface carries a nonconstant negative plurisubharmonic function \cite[Section 4]{ForstnericLarusson2012}.

\begin{exercise}  \label{x:inoue-hirzebruch}
We say that a complex manifold $Y$ is \textit{strongly Liouville} if the universal covering space of $Y$ carries no nonconstant negative plurisubharmonic functions.  Show that if $Y$ is either dominable or $\C$-connected, then $Y$ is strongly Liouville.
\end{exercise}

It follows that an Inoue-Hirzebruch surface is neither dominable nor $\C$-connected, and is therefore not Oka.

Thus Dloussky's example shows that the following properties are in general not closed in families of compact complex manifolds.
\begin{itemize}
\item  The Oka property.
\item  Strong dominability.
\item  Dominability.
\item  $\C$-connectedness.
\item  Strong Liouvilleness.
\end{itemize}
It may be that the only interesting closed anti-hyperbolicity property is the weakest anti-hyperbolicity property, the property of not being Kobayashi hyperbolic.

Dloussky's example also shows that the Brody reparametrisation lemma that is used to show that Kobayashi hyperbolicity is open in families of compact complex manifolds has no higher-dimensional version that could be used to similarly prove that being the target of a nondegenerate holomorphic map from $\C^2$ is closed in families.


\begin{thebibliography}{88}

\bibitem{Arnold1972}
Arnold, J. E.  \textit{Local to global theorems in the theory of Hurewicz fibrations.}  Trans. Amer. Math. Soc. \textbf{164} (1972) 179--188. 

\bibitem{Barth-et-al2004}
Barth, W. P., K. Hulek, C. A. M. Peters, and A. Van de Ven.  \textit{Compact complex surfaces.}  Second edition.  Ergebnisse der Mathematik und ihrer Grenzgebiete, 3. Folge, 4. Springer-Verlag, 2004.

\bibitem{Berndtsson2010}
Berndtsson, B.  \textit{An introduction to things $\bar\partial$.}  Analytic and algebraic geometry, 7--76, 
IAS/Park City Math. Ser., 17.  Amer. Math. Soc., 2010.

\bibitem{BrodyGreen1977}
Brody, R. and M. Green.  \textit{A family of smooth hyperbolic hypersurfaces in $P_3$.}  Duke Math. J. \textbf{44} (1977) 873--874. 

\bibitem{BuzzardLu2000}
Buzzard, G. T. and S. S. Y. Lu.  \textit{Algebraic surfaces holomorphically dominable by $\C^2$.}  Invent. Math. \textbf{139} (2000) 617--659.

\bibitem{CoxLittleSchenck2011}
Cox, D. A., J. B. Little, and H. K. Schenck.  \textit{Toric varieties.}  Graduate Studies in Mathematics, 124.  Amer. Math. Soc., 2011.

\bibitem{Demailly2012}
Demailly, J.-P.  \textit{Complex analytic and differential geometry.}  Book manuscript available at \texttt{http://www-fourier.ujf-grenoble.fr/\~{}demailly/manuscripts/agbook.pdf}  Version of 21 June 2012.

\bibitem{DiederichEtAl2013}
Diederich K., J. E. Fornaess, and E. F. Wold.  \textit{Exposing points on the boundary of a strictly pseudoconvex or a locally convexifiable domain of finite 1-type.}  J. Geom. Anal., to appear.  {\tt arXiv:1303.1976}

\bibitem{Dloussky2013}
Dloussky, G.  \textit{From non-K\"ahlerian surfaces to Cremona group of $\P^2(\C)$.} Complex Manifolds \textbf{1} (2014) 1--33. 

\bibitem{DrinovecForstneric2007}
Drinovec Drnov\v sek, B. and F. Forstneri\v c.  \textit{Holomorphic curves in complex spaces.}  Duke Math. J. \textbf{139} (2007) 203--253. 

\bibitem{DrinovecForstneric2012}
Drinovec Drnov\v sek, B. and F. Forstneri\v c.  \textit{The Poletsky-Rosay theorem on singular complex spaces.}  Indiana Univ. Math. J. \textbf{61} (2012) 1407--1423. 

\bibitem{Eremenko2006}
Eremenko, A.  \textit{Exceptional values in holomorphic families of entire functions.}  Michigan Math. J. \textbf{54} (2006) 687--696.

\bibitem{Forstneric2003}
Forstneri\v c, F.  \textit{Noncritical holomorphic functions on Stein manifolds.}  Acta Math. \textbf{191} (2003) 143--189.

\bibitem{Forstneric2005}
Forstneri\v c, F.  \textit{Extending holomorphic mappings from subvarieties in Stein manifolds.}  Ann. Inst. Fourier (Grenoble) \textbf{55} (2005) 733--751. 

\bibitem{Forstneric2006}
Forstneri\v c, F.  \textit{Runge approximation on convex sets implies the Oka property.}  Ann. Math. \textbf{163} (2006) 689--707.

\bibitem{Forstneric2009}
Forstneri\v c, F.  \textit{Oka manifolds.}  C. R. Math. Acad. Sci. Paris \textbf{347} (2009) 1017--1020.

\bibitem{Forstneric2010}
Forstneri\v c, F.  \textit{Oka maps.}  C. R. Math. Acad. Sci. Paris \textbf{348} (2010) 145--148. 

\bibitem{Forstneric2011}
Forstneri\v c, F.  \textit{Stein manifolds and holomorphic mappings.  The homotopy principle in complex analysis.}  Ergebnisse der Mathematik und ihrer Grenzgebiete, 3. Folge, 56.  Springer-Verlag, 2011.

\bibitem{Forstneric2013}
Forstneri\v c, F.  \textit{Oka manifolds: from Oka to Stein and back.} 
With an appendix by F. L\'arusson.  Ann. Fac. Sci. Toulouse Math. (6) \textbf{22} (2013) 747--809.

\bibitem{ForstnericLarusson2011} 
Forstneri\v c, F. and F. L\'arusson.  \textit{Survey of Oka theory.}  New York J. Math. \textbf{17a} (2011) 1--28. 

\bibitem{ForstnericLarusson2012}
Forstneri\v c, F. and F. L\'arusson.  \textit{Holomorphic flexibility properties of compact complex surfaces.}  Int. Math. Res. Not. IMRN \textbf{2014} 3714--3734. 

\bibitem{GoerssJardine1999}
Goerss, P. G. and J. F. Jardine.  \textit{Simplicial homotopy theory.}  Progress in Mathematics, 174.  Birkh\"auser, 1999.

\bibitem{Gromov1989}
Gromov, M.  \textit{Oka's principle for holomorphic sections of elliptic bundles.}  J. Amer. Math. Soc. \textbf{2} (1989) 851--897.

\bibitem{Hormander1979}
H\"ormander, L.  \textit{An introduction to complex analysis in several variables.}  Second edition.  North-Holland Mathematical Library, 7.  North-Holland, 1979.

\bibitem{KalimanKutzschebauch2008}
Kaliman, S. and F. Kutzschebauch.  \textit{Density property for hypersurfaces $UV=P(\bar X)$.}  Math. Z. \textbf{258} (2008) 115--131.

\bibitem{Kobayashi1998}
Kobayashi, S.  \textit{Hyperbolic complex spaces.}  Grundlehren der mathematischen Wissenschaften, 318. Springer-Verlag, 1998.

\bibitem{Larusson2004}
L\'arusson, F.  \textit{Model structures and the Oka principle.}  J. Pure Appl. Algebra \textbf{192} (2004) 203--223.

\bibitem{Larusson2005}
L\'arusson, F.  \textit{Mapping cylinders and the Oka principle.}  Indiana Univ. Math. J. \textbf{54} (2005) 1145--1159.

\bibitem{Larusson2009a}
L\'arusson, F.  \textit{Ellipticity and hyperbolicity in geometric complex analysis.}  Notes for three lectures given at a Winter School on Geometry and Physics at the University of Adelaide in July 2009.  Available at \texttt{http://www.maths.adelaide.edu.au/finnur.larusson}

\bibitem{Larusson2009b}
L\'arusson, F.  \textit{Affine simplices in Oka manifolds.}  Documenta Math. \textbf{14} (2009) 691--697.

\bibitem{Larusson2010a}
L\'arusson, F.  \textit{What is an Oka manifold?}  Notices Amer. Math. Soc. \textbf{57} (2010) 50--52.  

\bibitem{Larusson2010b}
L\'arusson, F.  \textit{Applications of a parametric Oka principle for liftings.}  In:  Complex analysis, pp.~205--211, Trends in Mathematics.  Birkh\"auser, 2010.

\bibitem{Larusson2012}
L\'arusson, F.  \textit{Deformations of Oka manifolds.}  Math. Z. \textbf{272} (2012) 1051--1058.

\bibitem{Larusson2013}
L\'arusson, F.  \textit{Absolute neighbourhood retracts and spaces of holomorphic maps from Stein manifolds to Oka manifolds.}  Proc. Amer. Math. Soc. \textbf{143} (2015) 1159--1167. 

\bibitem{May1992}
May, J.P.  \textit{Simplicial objects in algebraic topology.}  Reprint of the 1967 original.  Chicago Lectures in Mathematics.  University of Chicago Press, 1992. 

\bibitem{May1999}
May, J. P.  \textit{A concise course in algebraic topology.}  Chicago Lectures in Mathematics.  University of Chicago Press, 1999.

\bibitem{Nakamura1989}
Nakamura, I.  \textit{Towards classification of non-K\"ahlerian complex surfaces.}  Sugaku Expositions \textbf{2} (1989) 209--229.  Updated version available at {\tt http://www.math.sci.hokudai.ac.jp/\~{}nakamura}

\bibitem{RosayRudin1988}
Rosay, J.-P. and W. Rudin.  \textit{Holomorphic maps from $\mathbf{C}^n$ to $\mathbf{C}^n$.}  Trans. Amer. Math. Soc. \textbf{310} (1988) 47--86.

\bibitem{SmrekarYamashita2009}
Smrekar, J. and A. Yamashita.  \textit{Function spaces of CW homotopy type are Hilbert manifolds.}  Proc. Amer. Math. Soc. \textbf{137} (2009) 751--759.

\bibitem{Sobieszek2003}
Sobieszek, T.  \textit{On the Hartogs extension theorem.}  Proceedings of Conference on Complex Analysis (Bielsko-Bia\l a, 2001).  Ann. Polon. Math. \textbf{80} (2003) 219--222. 

\bibitem{vanMill1989}
van Mill, J.  \textit{Infinite-dimensional topology.  Prerequisites and introduction.}  North-Holland Mathematical Library, 43.  North-Holland Publishing Co., 1989.

\bibitem{Wold2010}
Wold, E. F.  \textit{A long $\C^2$ which is not Stein.}  Ark. Mat. \textbf{48} (2010) 207--210. 

\end{thebibliography}
\end{document}